\documentclass[11pt]{amsart}
\advance\textwidth 1cm
\advance\oddsidemargin -0.5cm
\advance\textheight 2cm
\advance\topmargin -1cm
\usepackage{amsmath}
\usepackage{amsfonts}
\usepackage{amssymb}

\newtheorem{lemma}{Lemma}[section]

\newtheorem{proposition}[lemma]{Proposition}

\newtheorem{corollary}[lemma]{Corollary}

\newtheorem{theorem}[lemma]{Theorem}

\makeindex

\title{A comparison of minimal systems for constructive analysis}

\author{Garyfallia Vafeiadou}

\thanks{The content of this paper (except the results of 8.1)
is from \cite{VafeiadouPhD}, written under the supervision of
Prof. Joan Rand Moschovakis, to whom the author is deeply grateful
for the invaluable privilege to learn from her, for the encouragement,
and for her numerous suggestions and comments on this work. The author
also thanks Iris Loeb for some motivating questions.}

\begin{document}

\begin{abstract}

We establish a precise relation between the \emph{minimal system
of analysis} {\bf M}, a subsystem of the formal axiomatic system
of intuitionistic analysis {\bf FIM} of S. C. Kleene, and
\emph{elementary analysis} {\bf EL} of A. S. Troelstra, two weak
formal systems of two-sorted intuitionistic arithmetic, both
widely used as basis for (various forms of) constructive analysis.
We show that {\bf EL} is weaker than {\bf M}, by introducing an
axiom schema $\mathrm{CF\!_d}$ asserting that every decidable (in
the sense that it satisfies the law of the excluded middle)
predicate of natural numbers has a characteristic function. As it
turns out, $\mathrm{CF\!_d}$ captures the essential difference of
the two minimal theories. Moreover, we obtain the conservativity
of {\bf M} over first-order intuitionistic arithmetic, and the
eliminability of Church's $\lambda$ from {\bf EL} by modifying the
proof of J. Rand Moschovakis of the corresponding result for {\bf
M}. We also show that {\bf EL}  and the formal theory {\bf BIM}
(\emph{Basic Intuitionistic Mathematics}) of W. Veldman are
essentially equivalent, and we compare some more systems of
two-sorted intuitionistic arithmetic.

\end{abstract}

\maketitle

\section*{Introduction}

The investigation in constructive analysis is carried out in a
multitude of formal or informal languages and systems, whose
relationships remain in many aspects unclear.  Starting an attempt
to elucidate the relations among the various formalisms used, we
compare particular formal systems of intuitionistic two-sorted
arithmetic that are neutral and formalize the common part of the
main varieties  of constructive analysis (corresponding to
Brouwer's intuitionism, Markov's Russian recursive mathematics and
Bishop's constructivism), and of classical analysis also.

Despite the different ways in which it can be understood,
constructivity affects the logic inherent in mathematical
reasoning,  as L. E. J. Brouwer, the founder of intuitionism,
realized and showed, leading to the rejection of the unrestricted
use of the law of the excluded middle. So all the systems that we
study  are based on intuitionistic (predicate) logic. In addition,
they have the following common features: they are formulated in
two-sorted languages, with variables for natural numbers and
one-place number-theoretic functions (in contrast to set variables
used for the second sort in the classical case, \cite{Simpson}),
and they have constants for (different selections of) only
primitive recursive functions and functionals, and equality
between natural numbers, with equality between functions defined
extensionally. $\lambda$-abstraction is included in some of them.
The function existence principles assumed are all weak and do not
involve real choice.

The differences in the languages as well as the interplay between
the possibilities provided by the languages and the assumed (if
any) function existence principles do not allow, in most cases, to
determine directly how these systems relate to each other. We
first obtain a precise relationship between the systems {\bf M}
and {\bf EL}, and then apply similar arguments to get comparisons
in some other cases.

\section{The formal systems {\bf M} and {\bf EL}}

\subsubsection{}

Starting with Heyting,  intuitionistic logic and arithmetic were
formalized as subsystems  of the corresponding classical formal
theories  (see \cite{JRM2009}). On the contrary, Heyting's
formalization of Brouwer's set theory (the part of intuitionistic
mathematics concerning the continuum and the real numbers) failed
to allow comparison with classical mathematics. S. C. Kleene and
R. E. Vesley  (\cite{FIM}) formalized large parts of
intuitionistic mathematics, corresponding to mathematical
analysis, in the formal system {\bf FIM} whose language is
suitable for classical analysis also, making such a comparison
possible.

The \emph{minimal system of analysis} {\bf M}, identified in
\cite{JRMPhD}, is a subsystem of {\bf FIM} consisting of the
primitive recursive arithmetical basis of {\bf FIM} and a
countable function comprehension principle. Within it S. C. Kleene
developed formally, with great detail, the theory of recursive
partial functionals (\cite{Kleene1969}). The corresponding
informal theory $\mathcal{M}$ is used in \cite{JRM2003}. We also
note that the system {\bf WKV} (Weak Kleene-Vesley) used in the
constructive reverse mathematics paper \cite{Loeb} is a
``minimalistic'' variant of {\bf M}.

The system of \emph{elementary analysis} {\bf EL}
(\cite{Troelstra1973, TvDI}) has been developed mostly by A. S.
Troelstra, to serve as a formal basis for intuitionistic analysis.
It differs from {\bf M} in its arithmetical  basis and in the
function existence principle it assumes. {\bf EL} is used in
recent work for the formalization of Bishop's constructive
analysis, especially in relation to the program of Constructive
Reverse Mathematics (see for example \cite{Berger} and
\cite{Ishihara}).

\subsubsection{}

{\bf M} and {\bf EL} have many similarities: all the common
features mentioned in the introductory paragraph, including
$\lambda$-abstraction, and also the possibility of definition by
primitive recursion, although with different justification.

Their differences are of two kinds. First, they have differences
in their languages. {\bf M} has only finitely many function and
functional constants, following the paradigm of (the usual
presentation of) Peano arithmetic, while {\bf EL} has infinitely
many function constants (not including the functional constants of
{\bf M}), as it extends (first-order intuitionistic) Heyting
arithmetic {\bf HA} (as presented in \cite{TvDI}), which contains
primitive recursive arithmetic {\bf PRA}; {\bf EL} has also a
recursor functional (not included in {\bf M}). Second, they assume
different function existence principles: {\bf M} assumes the axiom
schema $\mathrm{AC_{00}!}$ of countable function comprehension,
while {\bf EL} assumes the axiom schema
$\mathrm{QF\text{-}AC_{00}}$ of countable choice for
quantifier-free formulas, which is a consequence of
$\mathrm{AC_{00}!}$.

The two systems were considered more or less equivalent, but their
exact relation was unknown.\footnote{It is interesting that the
formal development of elementary recursion theory in {\bf EL} by
A. S. Troelstra ``relies heavily'' on the above mentioned work
\cite{Kleene1969}, see \cite{Troelstra1973}, p. 73.} We have found
that {\bf EL} is essentially weaker than {\bf M}, and that their
difference is captured  by the  function existence principle
expressed by an axiom schema which we call $\mathrm{CF\!_d}$,
another consequence of $\mathrm{AC_{00}!}$. After identifying
$\mathrm{CF\!_d}$, we show that {\bf EL} + $\mathrm{CF\!_d}$
entails $\mathrm{AC_{00}!}$, and that {\bf EL} does not entail
$\mathrm{CF\!_d}$. These results suggest that the formal system
{\bf EL} + $\mathrm{CF\!_d}$ is essentially equivalent to {\bf M},
while {\bf EL} is weaker than {\bf M}. In order to establish these
suggested relationships, we have to overcome the differences of
the languages. So we extend both systems up to a common language,
and we show that the corresponding extensions of {\bf M} and  {\bf
EL} + $\mathrm{CF\!_d}$  are conservative, in fact definitional,
and coincide up to trivial notational differences; so we conclude
that the suggested relationships hold indeed. We also show that,
like {\bf EL}, {\bf M} is conservative over first-order
intuitionistic arithmetic and that Church's $\lambda$ is
eliminable from {\bf EL}, by a slight modification of the
corresponding proof for {\bf M} given in \cite{JRMPhD}.

\subsection{Language and underlying logic}

\subsubsection{}

Both systems are based on two-sorted intuitionistic predicate
logic with equality, with number and function variables. Their
languages $\mathcal{L}(\mathrm{{\bf M})}$ and
$\mathcal{L}(\mathrm{{\bf EL})}$ have a common part which
includes: the logical symbols $\rightarrow ,~\& ~ ,\vee ,\neg
,\forall , \exists\,$, commas and parentheses as punctuation
symbols, number variables $\mathrm{x, y, z, \dots }$ intended to
range over natural numbers, and function variables $\alpha ,\beta
,\gamma ,\dots$ intended to range over one-place number-theoretic
functions (or choice sequences in the case of intuitionistic
analysis). The set of individual constants, predicate and function
and functional symbols of each language extends in different ways
a common part, and there are common and different formation rules
for the terms (type-0 terms, expressions for natural numbers) and
the functors (type-1 terms, expressions for one-place
number-theoretic functions (or choice sequences)); these will be
included in the description of the non-logical part of the
formalisms. The number equality predicate symbol = is contained in
both languages.

\subsubsection{}

The logical axioms and rules can be introduced in various ways,
for example in a natural deduction or Hilbert-type style. We will
base our treatment on the formal system of \cite{FIM} which
extends that of \cite{IM} (on p. 13 of \cite{FIM} and pp. 82 and
101 of \cite{IM}). The corresponding system of classical logic is
obtained by replacing axiom schema $\mathrm{\neg A\rightarrow
(A\rightarrow B)}$ of ``ex falso sequitur quodlibet'' by the
schema $\mathrm{\neg \neg A \rightarrow A}$ of double negation
elimination or, equivalently, by adding the schema $\mathrm{A \vee
\neg A}$ of the excluded middle.

\subsubsection{}

Number equality (between terms) is introduced as a primitive which
in {\bf HA} and {\bf EL} satisfies (the universal closures of) the
axiom $\;\mathrm{REFL \;\; x = x}\;$ and the replacement schema
$\;\mathrm{REPL\;\; A(x) ~\&~ x = y \rightarrow A(y)}$, where
$\mathrm{A(z)}$ is a formula and x, y are distinct number
variables free for z in $\mathrm{A(z)}$. In the version of
first-order arithmetic of \cite{IM} which we call $\mathrm{\bf
IA_0}$ and upon which {\bf M} is based, as well as in {\bf M},
these are reduced to a finite number of simple axioms, from which
REFL and the schema corresponding to REPL are provable.

Equality between functors, in all the systems that we consider, is
defined extensionally and is introduced by the abbreviation
$\mathrm{u = v \equiv \forall x \,(u) (x) = (v) (x)}$, where u, v
are functors, x is a number variable not free in u or v, (u)(x)
and (v)(x) are terms obtained by function application, according
to the formation rules given below. Details on the treatment of
equality will be given in a later section.

\subsection{Underlying arithmetic}

\subsubsection{}

The systems {\bf M} and {\bf EL} are based on weak systems of
two-sorted intuitionistic arithmetic, which we call $\mathrm{\bf
IA_1}$ and $\mathrm{\bf HA_1}$, respectively. Both weak theories
are based on the two-sorted intuitionistic predicate logic that we
described; they extend $\mathrm{\bf IA_0}$ and $\mathrm{\bf HA}$,
respectively. The difference between $\mathrm{\bf IA_0}$ and
$\mathrm{\bf HA}$ is that the first one has only finitely many
primitives, while the second has symbols for all the primitive
recursive functions. But it is well-known (see \cite{IM}, $\S$74,
for a detailed proof) that adding function symbols for primitive
recursive functions to $\mathrm{\bf IA_0}$ leads to definitional,
and so inessential extensions.

\subsubsection{}

{\bf IA$_0$} is based on first-order intuitionistic predicate
logic (with only number variables). Besides =, its language
$\mathrm{\mathcal{L}({\bf IA_0})}$ contains the constants 0
(zero), $'\,$ (successor),  + (addition) and $\,\cdot\,$
(multiplication).

Terms are defined inductively, as usual: the constant 0 and the
number variables are terms, and if s, t are terms, then
$\mathrm{(s)',\, (s)+(t),\, (s) \cdot (t)}$ are terms.

The prime formulas are the equalities between terms: if s, t are
terms, then $\mathrm{(s)=(t)}$ is a (prime) formula.

The mathematical axioms of {\bf IA$_0$} are the axioms for
$=,\,0,\, ', +, \cdot \,$ (on p. 82 of \cite{IM}) and the axiom
schema of mathematical induction
$$
\mathrm{IND\;\;\;\;\;\;\;A(0)~\&~ \forall x \,(A(x)\rightarrow
 A(x'))\rightarrow A(x)}.
$$

Classical first-order Peano arithmetic {\bf PA} is {\bf IA$_0$} +
$\mathrm{\neg \neg A \rightarrow A}$.

\subsubsection{}

Heyting arithmetic {\bf HA} differs from {\bf IA$_0$} in the set
of function constants it contains. Its language
$\mathrm{\mathcal{L}({\bf HA})}$ has the constant 0 (zero)  and
countably infinitely many constants $\mathrm{h_0, h_1, h_2, \ldots
}$, including S (successor), for all the primitive recursive
functions, more precisely a function constant for each primitive
recursive description.

The term formation rules are adapted accordingly: the constant 0
and the number variables are terms, and  if $\mathrm{t_1, \ldots
,t_{\textit{k}}}\,$  are terms and h a $k$-place function
constant, then $\mathrm{h(t_1, \ldots ,t_{\textit{k}})}$ is a
term.

The mathematical axioms of  {\bf HA}, besides the equality axioms
given by REFL and REPL, are the axiom schema of induction IND, the
axiom $\mathrm{\neg S(0) = 0}$ and defining axioms for the
function constants, which consist of the equations expressing the
corresponding primitive recursive descriptions.

\subsubsection{}

The language  $\mathrm{\mathcal{L}({\bf IA_1})}$ of {\bf IA$_1$}
(we note that $\mathrm{\mathcal{L}({\bf IA_1})}$ is
$\mathrm{\mathcal{L}({\bf M})}$) extends $\mathrm{\mathcal{L}({\bf
IA_0})}$. It has the (finitely many) function and functional
constants $\mathrm{f_0,\ldots , f_\textit{p}}$ given in a list
below, where each f$_i$ ($i=0, \ldots ,p$) has $k_i$ number
arguments and $l_i$ function arguments. All of them express
functions primitive recursive in their arguments. According to the
needs of the development of the theory, a different selection of
which function constants are included in the alphabet may be done,
in agreement with the intuitionistic view that no formal system
can exhaust the possibilities of  mathematical activity. The
particular formal system that we are considering, {\bf M},
contains the 27 function(al)s contained in the list; the system of
\cite{FIM} contains the first  25 of them, the last two have been
added in \cite{Kleene1969}. There are also parentheses serving as
constant for function application, and Church's $\lambda$ for
$\lambda$-abstraction.

The terms  and functors of {\bf IA$_1$} are defined by
simultaneous induction: 0 and the number variables are terms; the
function variables and each constant f$_i$ with $k_i$ = 1, $l_i$ =
0, are functors; if $\mathrm{t_1,\ldots
,t_{\textit{k}_\textit{i}}}$ are terms and $\mathrm{u_1,\ldots
,u_{\textit{l}_\textit{i}}}$ functors, then
$\mathrm{f_\textit{i}(t_1,\ldots
,t_{\textit{k}_\textit{i}},u_1,\ldots
,u_{\textit{l}_\textit{i}})}$ is a term; if u is a functor and t a
term, then (u)(t) is a term; if x is a number variable and t a
term, then $\lambda$x(t) (we also write $\lambda$x.t) is a
functor.

The prime formulas of {\bf IA$_1$} are the equalities
$\mathrm{(s)=(t)}$ where s, t are terms.

The mathematical axioms of {\bf IA$_1$} are: the axioms of {\bf
~IA$_0$} for $=,0,\,'\,,+, \cdot $, the axiom schema IND for
$\mathrm{\mathcal{L}({\bf IA_1})}$,  the equations expressing the
primitive recursive definitions of the additional function(al)
constants f$_4$ - f$_{26}$ [they are of the following forms,
corresponding to explicit definition and definition by primitive
recursion:
\begin{enumerate}
\item{$\;\;\; \mathrm{f_{\textit{i}}(y,a,\alpha )
 =  p(y,a,\alpha )}$,}
\item{$\left\{
\begin{array}{ll}
    \mathrm{f_{\textit{i}}(0,a,\alpha ) = q(a,\alpha )},\\
\mathrm{f_{\textit{i}}(y',a,\alpha ) =
r(y,f_{\textit{i}}(y,a,\alpha ),a,\alpha ),}
\end{array}\right.$ }
\end{enumerate}
where  $\mathrm{p(y,a,\alpha )}$, $\mathrm{q(a,\alpha )}$,
$\mathrm{r(y,z,a,\alpha )}$ are terms containing only the distinct
variables shown and only function constants from $\mathrm{f_0,
\ldots ,f_{\textit{i}-1}}$, and $\mathrm{y,a, \alpha }$ are free
for z in $\mathrm{r(y,z,a,\alpha )}$], the equality axiom for
function variables $\mathrm{x=y \rightarrow
 \alpha (x)=\alpha (y)}$,
and the axiom schema of $\lambda$-conversion $\mathrm{(\lambda
x.t(x))(s)=t(s)}$, where t(x) is any term and s is any term free
for x in t(x).

We next give the complete list of the function(al) constants of
$\mathrm{\mathcal{L}({\bf IA_1})}$ with their defining axioms,
where a, b are number variables.

\begin{enumerate}

\item[\boldmath$\cdot\;$\unboldmath f$_{0\;\,}$]
$\mathrm{\equiv \; 0}, \hfill k_0=0,\:l_0=0,\;$
\vskip 0.1cm

\item[\boldmath$\cdot\;$\unboldmath f$_{1\;\,}$]
$\mathrm{\neg\, a'=0,\:a=b\rightarrow a'=b',\;a'=b'\rightarrow
a=b}, \hfill k_1=1,\:l_1=0,\;$
\vskip 0.1cm

\item[\boldmath$\cdot\;$\unboldmath f$_{2\;\,}$]
$\mathrm{ a+0=a,\:a+b'=(a+b)'}, \hfill k_2=2,\:l_2=0,\;$
\vskip 0.1cm

\item[\boldmath$\cdot\;$\unboldmath f$_{3\;\,}$]
$\mathrm{a\cdot 0=0,\:a\cdot b'=a\cdot b+a}, \hfill
k_3=2,\:l_3=0,\;$
\vskip 0.1cm

\item[\boldmath$\cdot\;$\unboldmath f$_{4\;\,}$]
$\mathrm{a^0=1,\:a^{b'}\!=\,a^b\cdot a}, \hfill
k_4=2,\:l_4=0,\;$
\vskip 0.1cm

\item[\boldmath$\cdot\;$\unboldmath f$_{5\;\,}$]
$\mathrm{0!=1,\:(a')!=(a!)\cdot a'}, \hfill
k_5=1,\:l_5=0,\;$
\vskip 0.1cm

\item[\boldmath$\cdot\;$\unboldmath f$_{6\;\,}$]
$\mathrm{ pd(0)=0,\:pd(a')=a}, \hfill k_6=1,\:l_6=0,\;$ \vskip
0.1cm

\item[\boldmath$\cdot\;$\unboldmath f$_{7\;\,}$]
$\mathrm{a\-0=a,\:a\-b'=pd(a\-b)},\hfill
k_7=2,\:l_7=0,\;$
\vskip 0.1cm

\item[\boldmath$\cdot\;$\unboldmath f$_{8\;\,}$]
$\mathrm{min(a,b)=b\-(b\-a)},\hfill k_8=2,\:l_8=0,\;$ \vskip
0.1cm

\item[\boldmath$\cdot\;$\unboldmath f$_{9\;\,}$]
$\mathrm{max(a,b)=(a\-b)+b},\hfill k_9=2,\:l_9=0,\;$ \vskip
0.1cm

\item[\boldmath$\cdot\;$\unboldmath f$_{10}$]
$\mathrm{\overline{sg}(0)=1,\:\overline{sg}(a')=0}, \hfill
k_{10}\!=\!1,\:l_{10}\!=\!0,\;$ \vskip 0.1cm

\item[\boldmath$\cdot\;$\unboldmath f$_{11}$]
$\mathrm{sg(0)=0,\:sg(a')=1}, \hfill
k_{11}\!=\!1,\:l_{11}\!=\!0,\;$ \vskip 0.1cm

\item[\boldmath$\cdot\;$\unboldmath f$_{12}$]
$\mathrm{|a-b|=(a\-b)+(b\-a)},\hfill
k_{12}\!=\!2,\:l_{12}\!=\!0,\;$ \vskip 0.1cm

\item[\boldmath$\cdot\;$\unboldmath f$_{13}$]
$\mathrm{ rm(0,b)=0,}$

\item[] $\mathrm{rm(a',b)=(rm(a,b))'\cdot
 sg\,|b-(rm(a,b))'|},\hfill k_{13}\!=\!2,\:l_{13}\!=\!0,\;$
\vskip 0.1cm

\item[\boldmath$\cdot\;$\unboldmath f$_{14}$]
  $\mathrm{[\,0 / \, b\,]\,=0},$

\item[] $\mathrm{[\,a'/  \, b\,]\,=\,[\,a / \,
  b\,]\,+\overline{sg}\,|\,b-(rm(a,b))'\,|},\hfill
k_{14}\!=\!2,\:l_{14}\!=\!0,\;$
\vskip 0.1cm

\item[\boldmath$\cdot\;$\unboldmath f$_{15}$]
$\mathrm{f_{15}(0,\alpha )=0,\:f_{15}(z',\alpha )
 =f_{15}(z,\alpha )+\alpha (z)}, \hfill
 k_{15}\!=\!1,\:l_{15}\!=\!1,\;$

\item[] [\:alternative notation:
\;$\mathrm{\Sigma_{y<z}\alpha (y)\,]}$
\vskip 0.1cm

\item[\boldmath$\cdot\;$\unboldmath f$_{16}$]
$\mathrm{f_{16}(0,\alpha )=1,\:f_{16}(z',\alpha )=
  f_{16}(z,\alpha )\cdot\alpha (z)}, \hfill
  k_{16}\!=\!1,\:l_{16}\!=\!1,\;$

\item[] [\:alternative notation: \;$\mathrm
  {\Pi_{y<z}\alpha (y)\,}$] \vskip 0.1cm

\item[\boldmath$\cdot\;$\unboldmath f$_{17}$]
$\mathrm{f_{17}(0,\alpha )\!=\! \alpha (0),\:f_{17}(z',\alpha
  )\!=\! f_8(f_{17}(z,\alpha ),\alpha (z'))}, \hfill
  k_{17}\!=\!1,\:l_{17}\!=\!1,\;$

\item[] [\:alternative notation: \;$\mathrm{ min_{y\leq
  z}\alpha (y)\,]}$ \vskip 0.1cm

\item[\boldmath$\cdot\;$\unboldmath f$_{18}$]
$\mathrm{f_{18}(0,\alpha )\!=\!\alpha (0),\:f_{18}(z',\alpha
  )\!=\! f_9(f_{18}(z,\alpha ),\alpha (z'))}, \hfill
  k_{18}\!=\!1,\:l_{18}\!=\!1,\;$

\item[] [\:alternative notation: \;$\mathrm{ max_{y\leq
  z} \alpha (y)}\,]$\vskip 0.1cm

\item[\boldmath$\cdot\;$\unboldmath f$_{19}$]
$\mathrm{p_0=2,\, p_{i'}=\mu b_{b<p_i!+2}\,[\,p_i<b ~\&~
  Pr(b)\,]},\hfill  k_{19}\!=\!1,\:l_{19}\!=\!0,\;$\vskip
  0.2cm

\item[\boldmath$\cdot\;$\unboldmath f$_{20}$]
$\mathrm{(a)_i=\mu x_{x<a}\,[\,p_i^x\mid a ~\&~ \neg \,
 p_i^{x'}\mid a\,]}, \hfill
k_{20}\!=\!2,\:l_{20}\!=\!0,\;$\vskip
  0.2cm

\item[\boldmath$\cdot\;$\unboldmath f$_{21}$]
$\mathrm{lh(a)=\Sigma_{i<a}\,sg((a)_i)}, \hfill
  k_{21}\!=\!1,\: l_{21}\!=\!0,\;$\vskip 0.1cm

\item[\boldmath$\cdot\;$\unboldmath f$_{22}$]
$\mathrm{ a*b=a\cdot\Pi_{i<lh(b)}p_{lh(a)+i}^{(b)_i}}, \hfill
  k_{22}\!=\!2,\;l_{22}\!=\!0,\;$\vskip 0.1cm

\item[\boldmath$\cdot\;$\unboldmath f$_{23}$]
$\mathrm{\overline{\alpha }(x)=\Pi_{i<x}p_i^{\alpha (i)+1}},
  \hfill k_{23}\!=\!1,\:l_{23}\!=\!1,\;$\vskip 0.1cm

\item[\boldmath$\cdot\;$\unboldmath f$_{24}$]
$\mathrm{\widetilde{\alpha }(x)=\Pi_{i<x}p_i^{\alpha (i)}},
  \hfill k_{24}\!=\!1,\:l_{24}\!=\!1,\;$\vskip 0.1cm

\item[\boldmath$\cdot\;$\unboldmath f$_{25}$]
$\mathrm{a\circ b=\Pi_{i<max(a,b)}p_i^{max((a)_i,(b)_i)}},
  \hfill k_{25}\!=\!2,\:l_{25}\!=\!0,\;$\vskip 0.1cm

\item[\boldmath$\cdot\;$\unboldmath f$_{26}$]
$\mathrm{ccp(0)=1,\;\; ccp(y')=ccp(y)\cdot
  p_y^{r(y,\,ccp(y))}}, \hfill
  k_{26}\!=\!1,\:l_{26}\!=\!0,\;$\vskip 0.1cm

\end{enumerate}
   \vskip 0.1cm
where the following abbreviations are used:
 \begin{enumerate}
\item[\boldmath $\cdot$\unboldmath ]
$\mathrm{a<b} \equiv \mathrm{a'\- b=0}$; $\mathrm{a \leq
b}\equiv\mathrm{a<b \vee a = b}$; $\mathrm{a \mid
b}\equiv\mathrm{sg(rm(b,a))=0}$;
\item[\boldmath $\cdot$\unboldmath ]
$\mathrm{\mu
y_{y<z}\,R(y)}\equiv\mathrm{\Sigma_{x<z}\,\Pi_{y<x'}\,r(y)}$,
where r(y) is a term with \\ $\mathrm{\vdash R(y)
\leftrightarrow r(y)=0}$, $\mathrm{\vdash r(y) \leq 1}$ and x
not free in r(y);
\item[\boldmath $\cdot$\unboldmath ]
Pr(a) is a prime formula expressing that a is a prime number;
\item[\boldmath $\cdot$\unboldmath ]
 the term r(y,z) in $\mathrm{f_{26}}$ is constructed in
\cite{Kleene1969}, where computation tree numbers are
introduced to code partial recursive derivations.
\end{enumerate}

\textsc{Notation}. (i) In all the systems that we consider we will
use, unless otherwise stated, the same abbreviations for $<$ etc.
and the same symbols for (the same) functions and functionals as
in {\bf IA$_1$}.

(ii)  As in \cite{FIM}, we will use the following abbreviation
representing a primitive recursive coding of finite sequences of
natural numbers: for each $k\geq 0$, $\mathrm{\langle x_0, \ldots
,x_\textit{k}\rangle \equiv {\bf p}_0^{x_0} \cdot \ldots \cdot
{\bf p_{\textit{k}}}^{x_\textit{k}}},$ where $\mathrm{{\bf
p_{\textit{i}}}}$ is the numeral for the $i$-th prime $p_i$.

\subsubsection{}

The language $\mathrm{\mathcal{L}({\bf HA_1})}$ of {\bf HA$_1$}
(note that $\mathrm{\mathcal{L}({\bf HA_1})}$ is
$\mathrm{\mathcal{L}({\bf EL})}$) extends
$\mathrm{\mathcal{L}({\bf HA})}$. In addition to the function
constants of $\mathrm{\mathcal{L}({\bf HA})}$, there are
parentheses for function application and Church's $\lambda$ for
$\lambda$-abstraction, as in $\mathrm{\mathcal{L}({\bf IA_1})}$.
There is also a functional constant rec expressing the recursor
functional, which corresponds to definition by the schema of
primitive recursion.

The terms and functors of {\bf HA$_1$} are defined as in {\bf
IA$_1$}, with an additional term formation rule: if $\mathrm{t,
s}$ are terms and u a functor, then $\mathrm{rec(t, u, s)}$ is a
term.

The mathematical axioms of {\bf HA$_1$} are: the axioms of {\bf
HA} with IND and REPL extended to $\mathrm{\mathcal{L}({\bf
HA_1})}$, $\lambda$-conversion, and the following axioms for the
recursor constant rec:
$$
\mathrm{REC}\;\;\left\{
\begin{array}{ll}
    \mathrm{rec (t, u, 0)=t,} \\
    \mathrm{rec (t, u, S(s))= u (\langle rec(t, u, s),s\rangle ),}
\end{array}\right.
$$
where t, s are terms and u a functor.\footnote{The original
 formulation of REC  uses a pairing function j
  \emph{onto} the natural numbers,
but as it is remarked in \cite{Troelstra1973}, 1.3.9, they ``might
have used Kleene's $\mathrm{2^x\cdot 3^y}$''.}

\subsection{Function existence principles}

\subsubsection{}

The unique existential number quantifier $\mathrm{\exists !y}$ is
used to express the notion ``there exists a unique y such that
...'' and it is introduced as an abbreviation:
$$
\mathrm{\exists !y B(y) \equiv \exists y \,[\,B(y)~\&~
\forall z (B(z) \rightarrow y = z)\,]\,}.
$$

\subsubsection{}
The \emph{minimal system of analysis} {\bf M} is the  theory {\bf
IA$_1$} + AC$_{00}$!, with
$$
\mathrm{AC_{00}!\;\;\;\forall x
\exists !y A(x,y) \rightarrow \exists \alpha \forall x A(x,\alpha
(x)),}
$$
where x and $\alpha$ are free for y in $\mathrm{A(x,y)}$ and
$\alpha$ does not occur free in $\mathrm{A(x,y)}$.

The schema of unique choice AC$_{00}$! expresses a countable
function comprehension principle. Because of the uniqueness
condition in the hypothesis, there is no real choice. With
classical logic, it is equivalent to AC$_{00}$ (just like
AC$_{00}$! but without the ! in the hypothesis), expressing a
countable numerical choice principle. Constructively, as it is
shown in \cite{Weinstein} with a highly non-trivial proof,
AC$_{00}$! is weaker than AC$_{00}$.\footnote{A constructively
equivalent definition of $\mathrm{\exists !y B(y)}$ is by
$\mathrm{ \exists y B(y) ~\&~ \forall y \forall z (B(y) ~\&~ B(z)
\rightarrow y = z)}$ where the second conjunct expresses ``at most
one'', for which Bishop constructivists use $\mathrm{\forall y
\forall z(y \neq z \rightarrow (\neg B(y) \vee \neg B(z)))}$,
which is classically but not constructively equivalent: consider
the formula $\mathrm{B(x)\equiv (x=0 ~\& ~ P) \vee (x=1 ~\&~\neg
P)}$, where P is any formula and $\mathrm{x}$ any variable not
occurring free in P (\cite{JRM-GV2012}). However, any of the
alternatives could be used in the formulation of AC$_{00}!$, since
under the assumption $\mathrm{\exists x B(x)}$ they become
constructively equivalent.}

\subsubsection{}

\emph{Elementary analysis} {\bf EL} is the theory {\bf HA$_1$} +
 QF-AC$_{00}$, with
$$
\mathrm{QF\text{-}AC_{00}\;\;\;\forall x \exists y A(x,y)
\rightarrow \exists \alpha \forall x A(x,\alpha (x)),}
$$
where $\mathrm{A(x,y)}$ is a quantifier-free formula, in which x
is free for y and  $\alpha$ does not occur.

The schema QF-AC$_{00}$ expresses a weak principle of countable
numerical choice, for quantifier-free formulas. This principle
does not involve real choice either, since the quantifier-free
formulas are decidable, and in this case, existence entails
constructively unique existence (of the least such number). For
these basic facts we refer to section 2 below.

\section{Unique existence and decidability}

\subsection{}

In intuitionistic arithmetic unique existence (of a natural number
satisfying a predicate) and decidability (of natural number
predicates) are closely related.\footnote{In the classical case
all these are trivialities, as natural number existence always
entails unique existence of a least witness and every predicate is
decidable.} As a consequence, the principles $\mathrm{AC_{00}!}$
and $\mathrm{QF\text{-}AC_{00}}$ are  related in a precise manner,
over any reasonable two-sorted intuitionistic arithmetic. We give
next some results, most of them well-known, that provide basic
facts about the two notions. Using them we will determine how
$\mathrm{AC_{00}!}$ and $\mathrm{QF\text{-}AC_{00}}$ relate to
each other. In the following {\bf S} is any of {\bf IA$_0$}, {\bf
HA}, {\bf IA$_1$} or {\bf HA$_1$}, and $\;\vdash \;$ denotes
provability in {\bf S}.

\begin{lemma}
$\mathrm{\vdash \forall x \forall y (x=y \vee \neg x=y)}$.
\end{lemma}

\begin{lemma}
For any formula $\mathrm{A}$ of {\bf \;S} built up from the
formulas $\mathrm{P_1, \ldots ,P_{\textit{m}}}$ by propositional
connectives or bounded number quantifiers,
$$
\mathrm{ P_1 \vee
\neg P_1, \ldots , P_{\textit{m}} \vee \neg P_{\textit{m}} \vdash
A \vee \neg A}.
$$
\end{lemma}

\begin{lemma}
For any formula $\mathrm{A}$ of {\bf \;S} which is quantifier-free
or has only bounded number quantifiers (and no function
quantifiers), $\mathrm{ \vdash A \vee \neg A}$.
\end{lemma}

Although the least (natural) number principle fails in
intuitionistic arithmetic in general, it holds for number
predicates that are assumed decidable, and in this case the least
number is unique.

\begin{lemma}
In {\bf S},
$$
\mathrm{ \vdash \forall y (B(y) \vee \neg B(y)) \rightarrow
 \,[\, \exists y B(y) \rightarrow \exists ! y
(B(y) ~\&~ \forall z(z<y \rightarrow \neg B(z)))\,]\,}.
$$
\end{lemma}

The next lemma asserts that, conversely to the previous,
uniqueness entails decidability; it follows from the decidability
of number-theoretic equality.

\begin{lemma}
$\mathrm{\exists ! y B(y) \vdash B(y) \vee \neg B(y)}$.
\end{lemma}

The next lemma provides a fact very useful for our purposes.

\begin{lemma}
$\mathrm{\vdash A \vee \neg A \leftrightarrow \exists ! y \,[\,
y\leq 1 ~\&~ (y=0 \leftrightarrow A)\,]}$, with $\mathrm{A}$ not
containing $\mathrm{y}$ free.
\end{lemma}

\subsubsection{}

We can now draw a first immediate conclusion about  AC$_{00}$! and
$\mathrm{QF\text{-}AC_{00}}$.

\begin{proposition}
Over $\mathrm{\bf IA_1}$ (and $\mathrm{\bf HA_1}$)
$\mathrm{AC_{00}!}$ entails $\mathrm{QF\text{-}AC_{00}}$.
\end{proposition}

\begin{proof}
By Lemmas 2.3 and 2.4.
\end{proof}

\section{A characteristic function principle}

\subsection{The schema $\mathrm{CF\!_d}$}

\subsubsection{}

Consider the following schema, which asserts that every decidable
predicate of natural numbers has a characteristic function:
$$
\mathrm{CF\!_d \;\;\;\forall x (B(x) \vee \neg B(x))
\rightarrow \exists \beta \forall x \,[\,\beta(x)\leq 1 ~\&~
(\beta(x) = 0 \leftrightarrow B(x))\,]\,},
$$
where $\beta$ does not occur free in $\mathrm{B(x)}$.

Introducing this axiom schema allows to determine the exact
relation of $\mathrm{AC_{00}!}$ and $\mathrm{QF\text{-}AC_{00}}$;
and this in its turn will suggest the relation between {\bf M} and
{\bf EL}.

\begin{proposition}
Over $\mathrm{\bf IA_1}$ (and $\mathrm{\bf HA_1}$),
$\mathrm{AC_{00}!}$ entails $\mathrm{CF\!_d}$.
\end{proposition}

\begin{proof}
By Lemma 2.6.
\end{proof}

\subsubsection{}

Now we show that the unique choice principle $\mathrm{AC_{00}!}$
is equivalent to the conjunction of its two consequences
$\mathrm{QF\text{-}AC_{00}}$ and $\mathrm{CF\!_d}$ over
$\mathrm{\bf IA_1}$ and $\,\mathrm{\bf HA_1}$.

\begin{theorem}
Over $\mathrm{\bf IA_1}$ (and $\,\mathrm{\bf HA_1}$),
$\mathrm{QF\text{-}AC_{00}} + \mathrm{CF\!_d}$ entails
$\mathrm{AC_{00}!}$.
\end{theorem}

\begin{proof}
Assume $\mathrm{(a)\;\forall x \exists ! y A(x,y)}$. By Lemma 2.5
we get $\mathrm{\forall x \forall y \,[\,A(x,y) \vee \neg
A(x,y)\,]\,}$, so, by specializing for $\mathrm{(w)_0,(w)_1}$,
$\mathrm{\forall w \,[\,A((w)_0,(w)_1) \vee \neg
A((w)_0,(w)_1)\,]\,}$. Applying $\mathrm{CF\!_d}$ to this,
$\mathrm{\exists \beta \forall w \,[\,\beta(w)\leq 1 ~\&~
(\beta(w) = 0 \leftrightarrow A((w)_0,(w)_1))\,]\,}$, from which
(without $\exists \beta$, towards $\exists$-elim., specializing
for $\mathrm{\langle x, y \rangle}$) and (a) we get
$\mathrm{\forall x \exists y \beta (\langle x,y \rangle )=0}$. So
by  QF-AC$_{00}$, $\mathrm{\exists \alpha \forall x \beta(\langle
x,\alpha (x)\rangle )=0},$ and finally $\mathrm{\exists \alpha
\forall x A(x,\alpha (x))}$.
\end{proof}

\begin{corollary}
Over $\mathrm{\bf IA_1}$ (and $\mathrm{\bf HA_1}$),
$\mathrm{AC_{00}!}$ is equivalent to $\mathrm{QF\text{-}AC_{00}} +
\mathrm{CF\!_d}$.
\end{corollary}

\subsection{Classical models for weak theories of two-sorted
arithmetic}

\subsubsection{}

Let {\bf T} be the formal theory $\mathrm{\bf IA_1} +
\mathrm{QF\text{-}AC_{00}}$. {\bf T} can be extended to a
corresponding classical theory $\mathrm{ {\bf T}^\circ}$, by
replacing the axiom schema $\mathrm{\neg A \rightarrow (A
\rightarrow B)}$ by  $\mathrm{\neg \neg A \rightarrow A}$. We will
use $\mathrm{ {\bf T}^\circ}$  to see that {\bf T} does not prove
$\mathrm{CF\!_d}$ by showing that $\mathrm{ {\bf T}^\circ}$ has a
classical model in which $\mathrm{CF\!_d}$ fails.

\begin{theorem}
$\mathrm{(a)}$ $\mathrm{\bf IA_1} + \mathrm{QF\text{-}AC_{00}}$
does not prove $\mathrm{CF\!_d}$.

$\mathrm{(b)}$ $\mathrm{\bf EL}$ does not prove $\mathrm{CF\!_d}$.
\end{theorem}

\begin{proof}
(a) Let {\bf T} and $\mathrm{ {\bf T}^\circ}$ be as in the
discussion above. We consider the structure $\mathcal{GR}$ for the
language of $\mathrm{ {\bf T}^\circ}$ consisting of the sets and
functions given by (i)-(iii):

 (i) The set $\mathbb{N}$ of the natural
numbers, that serves as the universe of the first sort, over which
the number variables range.

(ii) The subset $\mathcal{GR}$\footnote{The structure that we
consider is characterized by the choice of the universe of the
second sort, so we use the same name for both.} of the set of all
functions from $\mathbb{N}$ to $\mathbb{N}$ consisting of all the
general recursive functions from $\mathbb{N}$ to $\mathbb{N}$,
which serves as the universe of the second sort, over which the
function variables range.

(iii) The function(al)s $f_0, \ldots, f_p$ that correspond to the
function(al) constants $\mathrm{f_0, \ldots, f}_p$: each $f_i$,
$i=0,\ldots ,p$, is the primitive recursive function(al) obtained
by the primitive recursive description expressed by the defining
axioms of $\mathrm{f}_i$.

The interpretation of a term or functor under an assignment into
$\mathcal{GR}$ and the notions of satisfaction and truth are as
usual. In particular, the interpretation
$\mathrm{u^{\mathcal{GR}}}$ in $\mathcal{GR}$ of a functor u of
the form $\mathrm {\lambda x.t}$ where t is a term, under an
assignment $v$, is given by
$$
\mathrm{(\lambda x. t)^{\mathcal{GR}}}=
\lambda n. \overline{v(\mathrm{x}|n)}(\mathrm{t}),
$$
where $\overline{v(\mathrm{x}|n)}$ is the extension to all terms
and functors of the assignment which assigns the natural number
$n$ to x and agrees with $v$ on all other variables, the $\lambda$
in the interpretation is the usual (informal) Church's $\lambda$,
and $n$ ranges over $\mathbb{N}$. Function application
(represented by parentheses) is interpreted accordingly. It is
straightforward that  $\mathcal{GR}$ is a model of $\mathrm{\bf
IA_1}$.

$\mathrm{QF\text{-}AC_{00}}$ holds in $\mathcal{GR}$: using the
following fact (shown in \cite{FIM}, pp. 27-31) and the least
number operator, we obtain the function asserted to exist by
$\mathrm{QF\text{-}AC_{00}}$.

\textsc{Fact.} For any formula Q which is quantifier-free (or has
only bounded number quantifiers), we can construct a term q, with
the same free variables as Q, such that $\mathrm{\vdash q \leq 1
\;\;\text{and} \;\;\vdash Q \leftrightarrow q = 0}$. The
construction of q and the proofs are done in $\mathrm{\bf IA_1}$.

It is easy to see that $\mathrm{CF\!_d}$ does not hold in
$\mathcal{GR}$, since the law of the excluded middle holds in
$\mathcal{GR}$ while e.g. the predicate $\exists y T(x,x,y)$,
where $T(x,y,z) \Leftrightarrow z$ is the code of the computation
of the value of the partial recursive function with g\"{o}del
number $x$ at the argument $y$ (the Kleene $T$-predicate), does
not have a general recursive characteristic function.

(b) The argument is similar to the one for (a). We only have to
consider now $\mathcal{GR}$ as a structure with infinitely many
functions, corresponding to the function constants for
number-theoretic functions of {\bf EL}, and the recursor
functional (which, we note, is itself a primitive recursive
functional).
\end{proof}

\textsc{Remark}. It is well-known that in the presence of
$\mathrm{AC_{00}!}$ Church's Thesis in the form $ \forall\alpha
\exists x \forall y \exists z (T(x,y,z) ~\&~ U(z) = \alpha (y))$,
where $T(x,y,z)$  is the Kleene $T$-predicate and $U$ the
result-extracting function, contradicts classical logic. The
previous theorem makes it clear that this is due to
$\mathrm{CF\!_d}$.

\begin{corollary}
$\mathrm{(a)}$ $\mathrm{\bf IA_1} + \mathrm{QF\text{-}AC_{00}}$ is
a proper subtheory of $\,\mathrm{\bf M}$.

$\mathrm{(b)}$ $\mathrm{\bf EL}$ is a proper subtheory of
$\,\mathrm{\bf EL}  + \mathrm{AC_{00}!} =
 \mathrm{\bf EL}  + \mathrm{CF\!_d}$.
\end{corollary}

\subsubsection{}

By interpreting the function variables as varying over all
primitive recursive functions of one number variable we obtain, as
in the previous theorem, a classical model for $\mathrm{\bf HA_1}$
in which $\mathrm{QF\text{-}AC_{00}}$ does not hold, as it
guarantees closure under the notion ``general recursive in''.
Similarly for $\mathrm{\bf IA_1}$.

\begin{theorem}
$\mathrm{(a)}$ $\mathrm{\bf HA_1}$ does not prove
$\mathrm{QF\text{-}AC_{00}}$.

$\mathrm{(b)}$ $\mathrm{\bf IA_1}$ does not prove
$\mathrm{QF\text{-}AC_{00}}$.
\end{theorem}

\begin{proof}
By using for example the primitive recursive characteristic
function of the Kleene $T$-predicate and the general but not
primitive recursive Ackermann function.
\end{proof}

\section{Introduction of a recursor in {\bf M}}

To show that {\bf M} and {\bf EL} + $\mathrm{CF\!_d}$ are
essentially equivalent we will find a common conservative, in fact
definitional, extension of both. In order to obtain it, we add one
by one the missing constants of each system, and show that the
corresponding extension is definitional. In this way we reach
conservative extensions of the two systems in the same language,
which are identical (except for trivial notational differences).
Our treatment is based on \cite{IM}, $\S$74, where the one-sorted
first-order case of definitional extensions is covered, and on
\cite{JRMPhD}, where the method is applied for a result in the
two-sorted case.

The first step is to add a recursor constant to {\bf M}. The
notions of conservative and of definitional extension will be our
main tool. We first give the definitions that we will use (see
also \cite{Troelstra1973}) and then make some useful observations
concerning equality and replacement.

\subsection{Conservative and definitional extensions}

\subsubsection{}

\textsc{Definition}. Let {\bf S$_1$}, {\bf S$_2$} be formal
systems based on (many-sorted) intuitionistic predicate logic with
equality, and let the language  $\mathcal{L}$({\bf S$_2$}) of {\bf
S$_2$} extend the language $\mathcal{L}$({\bf S$_1$}) of {\bf
S$_1$}, and the theorems of {\bf S$_2$} contain the theorems of
{\bf S$_1$}. {\bf S$_2$} is a {\em conservative extension} of {\bf
S$_1$} if the theorems of {\bf S$_2$} that are formulas of {\bf
S$_1$} are exactly the theorems of {\bf S$_1$}.

\textsc{Definition}. Let  {\bf S$_1$}, {\bf S$_2$} be formal
systems with $\mathcal{L}$({\bf S$_1$}) contained in
$\mathcal{L}$({\bf S$_2$}). {\bf S$_2$} is a {\em definitional
extension} of {\bf S$_1$} if there exists an effective mapping (or
translation) $\,'\,$ which, to each formula E of {\bf S$_2$},
assigns a formula E$'$ of {\bf S$_1$} such that:

\begin{enumerate}

\item[I.]{$\;\mathrm{E'  \equiv E,\;}$
for E a formula of $\mathcal{L}$({\bf S$_1$}).}

\item[II.]{ $\;\mathrm{\vdash_{S_2}  E'
\leftrightarrow E. }$}

\item[III.]{$\;$If $\mathrm{\Gamma \vdash_{S_2} E }$, then
$\mathrm{\Gamma ' \vdash_{S_1}  E' }$.}

\item[IV.]{ $\;\,'\,$ commutes with the logical operations (of {\bf S$_1$}).}

\end{enumerate}

If the addition of a  symbol gives a definitional extension, the
symbol is called \emph{eliminable (from the extended to the
original system)}; conditions I - IV are called \emph{elimination
relations}; and we say that the symbol is \emph{added
definitionally}.

A definitional extension is obviously conservative, and moreover
every theorem of the extended system is equivalent (in the
extended system), by a translation, to one of the original. So it
is an inessential extension.

\subsection{On equality and replacement}

\subsubsection{Many-sorted intuitionistic predicate logic with
equality}

The systems that we are studying are based on many-sorted (and
specifically two-sorted) intuitionistic predicate logic with
equality, so the following axiom and axiom schema should be
satisfied for each sort i (we refer to \cite{Troelstra1973}).
$$
\begin{array}{ll}
    \mathrm{REFL^i \;\;\;x^i = x^i,} \\
    \mathrm{REPL^i \;\;\;x^i = y^i \rightarrow
(A(x) \rightarrow A(y)), \;\text{with x, y free for}\; z^i\;\text{ in A(z)}.}
\end{array}
$$

\subsubsection{Treatment of  equality in the systems under
study}

In all the systems that we consider only number equality is given
as a primitive; function equality is defined extensionally by the
abbreviation $\mathrm{u = v \equiv \forall x \,(u) (x) = (v)
(x)}$.

By EQ we denote all the axioms REFL$^\text{i}$ and
REPL$^\text{i}$, i = 0, 1. It is possible (\cite{IM}, $\S$73), as
in the case of $\mathrm{\bf IA_0}$  and {\bf M}, to reduce the
axioms of EQ to simpler (and in some cases only finitely many)
axioms, as follows (we refer to systems with only function(al)
constants; the case of predicate constants is treated similarly):

(A) By equality axioms for the binary predicate symbol $\;=\;$ are
meant the axioms $\mathrm{ x=x,}$ and $\mathrm{x=y \rightarrow
(x=z \rightarrow y=z)}$.

(B) By equality axioms for a function(al) symbol f of $k$ number
and $l$ function arguments are meant the $k$ formulas
$$
\begin{array}{rl}
    \mathrm{x=y \rightarrow} & \mathrm{f(x_1, \ldots ,
    x_{\textit{i}-1},x,
    x_{\textit{i}+1}, \ldots ,x_{\textit{k}}, \alpha _1,
    \ldots , \alpha _{\textit{l}})=}\\
    &\mathrm{ f(x_1, \ldots , x_{\textit{i}-1},y,
    x_{\textit{i}+1}, \ldots ,x_{\textit{k}}, \alpha _1,
    \ldots , \alpha _{\textit{l}}),\;\; \;\textit{i} = 1,
   \ldots , \textit{k},}
\end{array}
$$
and the corresponding $l$ formulas for function variables.

(C) The axioms EQ of a two-sorted formal system with type-0
equality as a primitive and type-1 equality defined as above, and
with only function(al) constants, are provable from
the following instances or consequences of them:\\
1. The equality axioms for = .\\
2. The equality axioms for the function(al) constants of its
alphabet.\\
3. The equality axiom for function variables $\mathrm{ x = y
\rightarrow \alpha (x) = \alpha (y)}$.

Thanks to the fact that the function(al) constants of {\bf M} (and
$\mathrm{\bf IA_1}$) are introduced successively via the primitive
recursive description of the corresponding functions, the equality
axioms for these are provable in {\bf M}; the proofs are by use of
IND (see Lemma 5.1 on p. 20 of \cite{FIM}). In the case of {\bf
EL} (and $\mathrm{\bf HA_1}$), the axioms by REFL$^0$, REPL$^0$
are all introduced from the beginning, but it is easy to see that,
in this case too, it suffices to include (C) 1, 3.

\subsubsection{The replacement theorem}

Lemma 4.2, p. 16 of \cite{FIM}, gives the replacement theorem for
{\bf M}. Since the proviso of the lemma is satisfied in the case
of {\bf EL} as a consequence of the preceding paragraph, Lemma 4.2
of \cite{FIM} provides the replacement theorem for {\bf EL}. The
same holds for the system that we will obtain by adding a recursor
to {\bf M}.

We note that the replacement theorem requires the equality axioms
only for the function symbols that have the specified occurrence
to be replaced within their scope, in the formula in which the
replacement takes place.

\subsection{Introducing a recursor in {\bf M}}

\subsubsection{}

We will add now to {\bf M} a recursor functional and  prove that
the resulting extension is definitional. Let {\bf S$_1$} be the
\textit{minimal system of analysis} {\bf M} and {\bf S$_2$}  the
system {\bf M} + Rec, obtained by adding to {\bf M} the functional
constant rec  together with the corresponding term formation rule
``if t, s are terms and u a functor, then rec(t, u, s) is a
term'', and the following axiom Rec defining it:
$$
\mathrm{Rec\;\;\;\;A(x, \alpha ,y, rec(x, \alpha ,y) ),}
$$
where $\mathrm{A(x, \alpha ,y,w)}$ is the formula
$$
\mathrm{\exists \beta \,[\,\beta (0)=x ~\&~ \forall z \,\beta
(z')=\alpha (\langle \beta (z),z\rangle ) ~\&~ \beta (y)=w
\,]\,.}
$$
The new constant rec represents then the recursor
functional, which corresponds to definition by the schema of
primitive recursion.

\vskip 0.1cm

\textsc{Remark}. We could have introduced the new functional
constant rec in {\bf M}  by the pair of equations REC that define
it in {\bf EL} and consider it as the f$_{27}$, extending the list
of constants of {\bf M} (using the second of the forms of the
definitions of the constants f$_i$). We have not adopted this
choice, because the presence of rec would make redundant many of
the constants of the list and because we consider this addition
temporary, only for the purpose of comparison.

Lemma 5.3(b) of \cite{FIM} (stated below) provides definition by
primitive recursion in {\bf M}, so we based our definition
directly on it. We followed \cite{IM}, $\S$74, in introducing a
new function symbol by a formula for which the formalism proves
that it has a functional character. We note also that some of the
formal  systems that we will consider ({\bf BIM}, {\bf WKV}, {\bf
H}) have  definition by primitive recursion as an axiom or axiom
schema, in forms very similar to Lemma 5.3(b) of \cite{FIM}. As we
will see the two ways of introducing  the new constant are
equivalent.

\subsubsection{Interderivability of $\mathrm{Rec}$
  and $\mathrm{REC}$}

We can easily see that {\bf S$_2$} = {\bf M} + Rec is equivalent
with {\bf S$'_2$} = {\bf M} + REC, in the sense that every
instance of REC is provable in {\bf S$_2$} and vice versa. In
fact, we can show that REC and Rec are interderivable over
$\mathrm{\bf IA_1}$ (and $\mathrm{\bf HA_1}$). So, in
 {\bf EL} and other systems, it is  immaterial which of
the two ``definitions'' of the constant rec is considered.

\subsubsection{}

We will show that {\bf S$_2$} is a definitional extension of {\bf
S$_1$}.

\textsc{Notation.} (a) In the following, by $\vdash _1$ and
$\vdash _2$ we denote provability in {\bf S$_1$} and {\bf S$_2$},
respectively.\\
(b) The unique existential function quantifier is introduced as an
abbreviation by
$$
\mathrm{\exists ! \beta C(\beta ) \equiv
 \exists \beta \,[\,C(\beta ) ~\&~ \forall \gamma \,
 (C(\gamma )
 \rightarrow \beta = \gamma )\,]}.
$$

\textsc{Remark 1}. With the help of the unique existential
function quantifier, we can formulate compactly the following
version of AC$_{00}$!:
$$
\mathrm{\forall x \exists !y A(x,y)
\rightarrow \exists ! \alpha \forall x A(x,\alpha (x)).}
$$

Although this schema is apparently stronger than AC$_{00}$!, it is
easily shown that it is a consequence of it, hence equivalent
(over two-sorted intuitionistic logic with equality). We will use
this version in some proofs.

The  following lemma (\cite{FIM}, p. 39) is proved in {\bf M} and
justifies ``definition by primitive recursion'' in this formal
theory.

\textsc{Lemma} 5.3(b) (\cite{FIM}). Let $\mathrm{y, z}$ be
distinct number variables, and $\alpha $ a function variable. Let
$\mathrm{q,\, r(y, z)}$ be terms not containing $\alpha$ free,
with $\alpha$ and $\mathrm{y}$ free for $\mathrm{z}$ in
$\mathrm{r(y, z)}$. Then
$$
\mathrm{\vdash \exists \alpha \,[\, \alpha (0) = q
~\&~ \forall y\,
\alpha (y') = r(y, \alpha (y))\,]\,}.
$$

The two next lemmas are easy consequences of \textsc{Lemma} 5.3(b)
(\cite{FIM}).

\begin{lemma}
$\mathrm{\vdash _1 \exists ! \beta  \,[\, \beta (0) = x ~\&~
\forall z \,\beta (z')=\alpha (\langle \beta (z),z \rangle
)\,]\,}$.
\end{lemma}

\begin{lemma}
$\mathrm{\vdash _1 \forall y \exists ! m \, A(x, \alpha , y, m)}$.
\end{lemma}

\begin{lemma}
$\mathrm{\vdash _2 rec(x, \alpha ,y )=z \leftrightarrow A(x,
\alpha , y, z)}$.
\end{lemma}

\begin{proof}
The formula $\mathrm{(a)\;\;A(x, \alpha, y, rec(x, \alpha , y))}$
is an axiom of {\bf S$_2$}.

(i) Assume $\mathrm{ rec(x, \alpha ,y )=z}.$ From this, (a) and
the replacement property of equality (which requires only the
predicate calculus with the equality axioms for = and the function
symbols of $\mathrm{ A(x, \alpha , y, z)}$) we get $\mathrm{ A(x,
\alpha , y, z)}$.

(ii) Assuming $\mathrm{ A(x, \alpha , y, z)}$, from  (a) with
Lemma 4.2 we obtain $\mathrm{rec(x, \alpha ,y )=z }$.
\end{proof}

\textsc{Remark 2}. The equality axioms for rec become now provable
from the above lemmas, or alternatively from the (equivalent)
definition of rec by REC, by the method of \cite{FIM}, Lemma 5.1
(see 4.2.2, on the treatment of equality).

\vskip 0.1cm

\textsc{Notation}. (i) Let $\mathrm{t}$ be a term. Let all the
free number variables\footnote{Note that the $\lambda$-prefixes
$\lambda$x, where x is any number variable, bind number
variables.} of $\mathrm{t}$ be among $\mathrm{x_0, \ldots , x}_k$,
and let $\mathrm{w}$ be a number variable not occurring in
$\mathrm{t}$. We will write $\mathrm {t^w}$ for the result of
replacing in $\mathrm{t}$, for each $i=0, \ldots ,k$, each free
occurrence of $\mathrm{x}_i$ by an occurrence of the term
$\mathrm{(w)}_i$. The same notation will be used for functors too.
Since the exponential will not appear in the proofs, there is no
chance of confusion by the use of this notation.

(ii) By the notation A(t) we represent as usual the result of
substituting a term t for all the free occurrences of x in a
formula A(x), and we tacitly assume that, if needed, some bound
variables are renamed, so that the substitution becomes free.
Similarly for functors, and also for many ``arguments''.

\begin{lemma}
Let $\mathrm{t, s}$ be terms and $\mathrm{u}$ a functor of
$\,\mathrm{{\bf S_1}}$. Let $\mathrm{x_0, \ldots , x}_k$ include
all the number variables occurring free in $\mathrm{t, u
\;\text{or} \;s}$, let $\,\mathrm{w}$ and  $\mathrm{v}$ be
distinct number variables not occurring in $\mathrm{t, u, s}$ and
$\gamma$ a function variable free for $\mathrm{v}$ in
$\mathrm{A(t^w, u^w, s^w, v)}$, not occurring free in
$\mathrm{A(t^w, u^w, s^w, v)}$. Then
$$
\mathrm{\vdash _1 \exists ! \gamma \forall w A(t^w,
u^w, s^w, \gamma (w))}.
$$
\end{lemma}

\begin{proof}
Let x, y, z be distinct number variables different from
$\mathrm{w,x_0, \ldots , x}_k$ and $\alpha$ a function variable
not occurring free in t, free for $\delta$ in $\mathrm{A(x,\delta
,y, z)}$. By Lemma 4.2 we get $\mathrm{\vdash _1 \forall x \forall
\alpha \forall y \exists ! z \, A(x, \alpha , y, z)},$ so, after
specializing for t, u, s with the corresponding
$\forall$-eliminations, we get
 $\mathrm{\vdash _1
\forall x_0 \ldots \forall x_{\text{k}} \exists ! z \, A(t, u , s,
z)}.$ From this, after specializing for each $i=0, \ldots ,k$ for
(w)$_i$, we get $\mathrm{\vdash _1 \forall w \exists ! z \, A(t^w,
u^w, s^w, z)},$ and by AC$_{00}$! with \textsc{Remark 1} of 4.3.3
we get $\mathrm{\vdash _1 \exists ! \gamma \forall w A(t^w, u^w,
s^w, \gamma (w))}.$
\end{proof}

\textsc{Notation}. (i) We use the notation $\mathrm{E[\,a\,]}$ to
indicate some specified occurrences of a term or functor a in an
expression E. We will also use similarly $\mathrm{E[\,a_1, \ldots
,a_{\textit{k}}\,]}$ for $k> 0$ to indicate  some specified
occurrences of $k$ terms or functors. This notation may leave some
ambiguity regarding the indicated occurrences, but in each case we
will explain its use.

(ii) We use $\mathrm{\alpha (x_0, \ldots , x}_k)$ as an
abbreviation for $\mathrm{\alpha (\langle x_0, \ldots ,
x}_k\rangle )$.

\begin{lemma}
Let $\mathrm{t,s}$ be terms and $\mathrm{u}$ a functor of
$\,\mathrm{\bf S_2}$, let $\mathrm{x_0, \ldots , x}_k$ be all the
number variables occurring free in $\mathrm{t, u \;\text{or}\;
s}$, and let $\mathrm{w}$ be a number variable not occurring in
$\mathrm{A(t, u, s,v)}$. Let $\mathrm{E[\,\gamma (x_0, \ldots ,
x}_k)\,]$ be a formula of $\,\mathrm{\bf S_2}$ in which
$\mathrm{\gamma (x_0, \ldots , x}_k)$ is  not within the scope of
some function quantifier $\forall \alpha$ or $\exists \alpha$,
where $\alpha$ is a function variable occurring free in
$\mathrm{t, u \;\text{or} \;s}$ or $\alpha$ is $\gamma$, and
$\gamma $ is new for $\mathrm{E[\,rec(t, u, s)}\,]$, is free for
$\mathrm{v}$ in $\mathrm{A(t^w, u^w, s^w, v)}$, and does not occur
free in $\mathrm{A(t^w, u^w, s^w, v)}$, and  where
$\mathrm{E[\,rec(t, u, s)}\,]$ is obtained by replacing in
$\mathrm{E[\,\gamma (x_0, \ldots , x}_k)\,]$ each of the
(specified) occurrences of $\,\mathrm{\gamma (x_0, \ldots , x}_k)$
by an occurrence of  $\,\mathrm{rec(t, u, s)}$. Then
$$
\mathrm{\vdash _2 E[\,rec(t, u, s)\,] \leftrightarrow
\exists \gamma [\,\forall w  A(t^w, u^w, s^w,\gamma (w))
~\&~ E[\,\gamma (x_0, \ldots ,x}_k)\,]\,].
$$
\end{lemma}

\begin{proof}
(i) Assume  $\mathrm{(a)\;E[\,rec(t, u, s)\,]}$. We have
$\mathrm{(b)\;\vdash _2 \exists  \gamma \forall w A(t^w, u^w, s^w,
\gamma (w))}$ by Lemma 4.4, so we assume $\mathrm{(c)\;\forall w
A(t^w, u^w, s^w, \gamma (w))}.$ Now by Lemma 4.3 we have
$\mathrm{\forall w\, rec(t^w, u^w, s^w) = \gamma (w)},$ so by
specializing for $\mathrm{ \langle x_0, \ldots , x}_k\rangle$ we
get
$$
\mathrm{(d)\;\; \forall x_0 \ldots  \forall x_{\text{k}}\; rec(t,
u, s) = \gamma (x_0, \ldots , x}_k).
$$
Since the occurrences of $\mathrm{\gamma (x_0, \ldots , x}_k)$ are
not within the scope of some function quantifier $\forall \alpha$
or $\exists \alpha$ where $\alpha$ occurs free in s, t, or u or
$\alpha$ is $\gamma$, by the replacement theorem, from (a) we get
$\mathrm{(e)\;E[\,\gamma (x_0, \ldots , x}_k)\,],$ so with (c) and
$\&$-introduction and then $\exists\gamma$-introduction and
$\exists\gamma$-elimination discharging (c), after
$\rightarrow$-introduction we get the ``$\rightarrow$'' case from
(a).

(ii) Assume $\mathrm{(a)\;\;\forall w A(t^w, u^w, s^w, \gamma
(w))}$ and $\mathrm{(b)\;\;E[\,\gamma (x_0, \ldots , x}_k)\,].$ By
(a), specializing for $\mathrm{ \langle x_0, \ldots ,
x}_k\rangle$, we get $\mathrm{A(t, u, s, \gamma (x_0, \ldots ,
x}_k)),$ and by Lemma 4.3 we get $\mathrm{(c)\;\; \forall x_0
\ldots  \forall x_{\text{k}} \,rec(t, u, s) = \gamma (x_0, \ldots
, x}_k) .$ By the conditions on the bindings due to function
quantifiers, the replacement theorem applies and from (b), (c) we
get $\mathrm{E[\,rec(t, u, s)\,]}$, and after
$\exists\gamma$-elimination discharging (a) and (b), and with
$\rightarrow$-introduction, we get the ``$\leftarrow$'' case.
\end{proof}

\textsc{Terminology}. A term of the form $\mathrm{rec(t,u,s)}$ is
called a \emph{rec-term}. A term in which the constant rec does
not occur is called a \emph{rec-less term}. A term $\mathrm{rec(t,
u, s)}$ where rec does not occur in $\mathrm{t,u,s}$ is called a
\emph{rec-plain term}. An occurrence of the constant rec in a
formal expression is called a \emph{rec-occurrence}.

\begin{lemma}
To each formula $\mathrm{E}$ of $\,\mathrm{\bf S_2}$ there can be
correlated a formula $\mathrm{E'}$ of  $\,\mathrm{\bf S_1}$,
called the principal rec-less transform of $\,\mathrm{E}$, in such
a way that the elimination relations $\mathrm{I, II}$ hold, no
free variables are introduced or removed, and the logical
operators of the two-sorted predicate logic are preserved
(elimination relation $\mathrm{IV}$).
\end{lemma}

\begin{proof}
The definition of E$'$ is done by induction on the number $g$ of
occurrences of the logical operators in E. The basis of the
induction consists in giving the definition for E prime; this is
done by induction on the number $q$ of occurrences of rec-terms in
E.

\textsc{Case} E is rec-less: then E$'$ shall be E.

\textsc{Case} E has $q>0$ rec-occurrences: let $\mathrm{rec(t, u,
s)}$ be the first (the leftmost) occurrence of a rec-plain term,
so that $\mathrm{E \equiv E[\,rec(t, u, s)}\,]$, and let
$\mathrm{x_0, \ldots , x}_k$ be all the free number variables of
t, u, s, w a number variable and $\gamma$ a function variable new
for both $\mathrm{E \equiv E[\,rec(t, u, s)}\,]$ and $\mathrm{A(t,
u, s, v)}$. Then we define
$$
\mathrm{E' \equiv
\exists \gamma [\,\forall w  A(t^w, u^w, s^w,\gamma (w)) ~\&~
[\,E[\,\gamma (x_0, \ldots ,x}_k)\,]\,]'\,],
$$
where $\mathrm{E[\,\gamma (x_0, \ldots , x}_k)\,]$ is the result
of replacing in E the specified (first) occurrence  of
$\,\mathrm{rec(t, u, s)}$ by an occurrence of $\mathrm{\gamma
(x_0, \ldots , x}_k)$. Then $\mathrm{E[\,\gamma (x_0, \ldots ,
x}_k)\,]$ is prime and contains $q-1$ occurrences of rec-terms,
and w is the only number variable free in $\mathrm{ A(t^w, u^w,
s^w,\gamma (w))}$. About the choice of the bound variables
$\gamma$ and w and the possibly necessary changes of the bound
variables of A(x, $\alpha$, y, v) to make the substitutions of
t$\mathrm{^w}$, u$\mathrm{^w}$, s$\mathrm{^w}$ free, all
permissible choices lead to congruent formulas.

The condition that the logical operators are preserved, i.e.
$\mathrm{(\neg \,A)' \equiv \neg \,(A')}$, $\mathrm{(A \circ B)'
\equiv A' \circ B',\;}$ for $\;\mathrm{\circ \equiv \;\rightarrow,
\& , \vee,\; }$ and $\;\mathrm{(Qx A(x))' \equiv Qx(A(x))',\;}$
for $\;\mathrm{Q \equiv \;\forall ,\exists }$, $\mathrm{(Q\alpha
A(\alpha ))' \equiv Q\alpha (A(\alpha ))'}$ for $\mathrm{Q \equiv
\;\forall ,\exists }$, determines in a unique way the definition
of $\,'\,$ for all formulas of $\mathrm{\bf S_2}$.

We immediately see that elimination relations I and IV hold.
Elimination relation II is now proved easily by induction on the
number of logical operators in E. The basis of the induction is
the case of E prime, and is proved (easily) by induction on the
number of rec-occurrences in E, using Lemma 4.5 and replacement.
\end{proof}

We still need to prove elimination relation III, so we have to
show: if $\mathrm{\Gamma \vdash_{S_2} E }$, then $\mathrm{\Gamma '
\vdash_{S_1}  E' }$. The proof depends on a sequence of lemmas
that follow. We will use the version with function variables of
Lemma 25 \cite{IM}, p. 408. This lemma provides useful facts, most
of them consequences of unique existence, for number variables, in
intuitionistic predicate logic with equality. Corresponding
results have been obtained by S. C. Kleene  for function
variables, in the two-sorted case (in a manuscript mentioned in
\cite{JRMPhD}); we can use this version thanks to Lemma 4.4. The
versions of \cite{IM} $^*$181 - $^*$190 with a function variable
instead of v are mentioned as $\mathrm{^*181^F}$ -
$\mathrm{^*188^F}$, $\mathrm{^*189_n^F}$, $\mathrm{^*190_n^F}$,
and there are also cases $\mathrm{^*189_f^F}$,
$\mathrm{^*190_f^F}$, $\mathrm{^*189_f^N}$, $\mathrm{^*190_f^N}$,
with variables whose sorts are obvious from the notation. We will
state the needed cases in the places we use them.

\begin{lemma}
Let $\mathrm{rec(t, u, s)}$ be any specified occurrence of a
rec-plain term in a prime formula $\,\mathrm{E}$ of $\,\mathrm{\bf
S_2}$, so that we have $\mathrm{E \equiv E[\,rec(t, u, s)}\,]$.
Then
$$
\mathrm{(a) \;\;\vdash _1 E' \leftrightarrow \exists \gamma
[\,\forall w  A(t^w, u^w, s^w,\gamma (w)) ~\&~ [\,E[\,\gamma (x_0,
\ldots ,x}_k)\,]\,]'\,],
$$
where the conditions on the variables are as in the definition of
$\;'$.
\end{lemma}

\begin{proof}
The proof is by induction on the number $q$ of occurrences of
rec-terms in the prime formula E. We will use the functional
version of $^*78$, \cite{IM}, p. 162,
$$
\mathrm{^*78_f^F\;\; \vdash \exists \alpha \exists \beta
\,D(\alpha , \beta ) \leftrightarrow \exists \beta \exists \alpha
\,D(\alpha , \beta )},
$$
and the following case of the functional version of  Lemma 25 of
\cite{IM},
\begin{multline*}
\mathrm{^*190_f^F\;\; \vdash \exists  \beta \,[\,F(\beta ) ~\&~
\exists \alpha \,D(\alpha , \beta )\,] \leftrightarrow \exists
\alpha \exists \beta \,[\, F(\beta ) ~\&~  D(\alpha , \beta
)\,],}\\ \text{where}\; \alpha \;\text{does not occur  free in}
\;\mathrm{F(\beta )}.
\end{multline*}
Cases $q=0$ or $q=1$ are trivial. For $q>1$, assume that Lemma 4.7
holds for prime formulas having $q-1$ rec-occurrences, and let E
be a prime formula with $q$ rec-occurrences. Let $\mathrm{rec(t,
u, s)}$ be a specified occurrence of a rec-plain term in E, so
$\mathrm{E \equiv E\,[\,rec(t, u, s)}\,]$. If $\mathrm{rec(t, u,
s)}$ is the first occurrence of a rec-plain term in E, then (a)
holds by the definition of $\,'\,$. If $\mathrm{rec(t, u, s)}$ is
not the first occurrence of a rec-plain term, then let
$\mathrm{rec(t_1, u_1, s_1)}$ be the first such, so that
$\mathrm{E \equiv E\,[\,rec(t_1, u_1, s_1), rec(t, u, s)\,]}$.
Then, by the definition of $\,'\,$, we have
$$
\mathrm{(b)\;\;E' \equiv
\exists \gamma [\,\forall w  A(t_1^w, u_1^w, s_1^w,\gamma (w))
~\&~ [\,E[\,\gamma (y_0, \ldots ,y}_l),\mathrm{rec(t, u,
s)}\,]\,]'\,],
$$
where $\mathrm{y_0, \ldots ,y}_l$ are the free number variables of
$\mathrm{t_1,u_1, s_1}$ and $\gamma$, w as in the definition of
$\,'\,$. By the inductive hypothesis, since $\mathrm{E[\,\gamma
(y_0, \ldots , y}_l),\mathrm{rec(t, u, s)}\,]$ has $q-1$
rec-occurrences, with the replacement theorem, from (b) we get
\begin{multline*}
\mathrm{(c)\;\;\vdash _1 E' \leftrightarrow \exists \gamma
[\,\forall w A(t_1^w, u_1^w, s_1^w,\gamma (w)) ~\&~ \exists \delta
[\,\forall w
A(t^w, u^w, s^w,\delta (w))}\\
\mathrm{ ~\&~ [\,E[\,\gamma (y_0, \ldots ,y}_l), \mathrm{\delta
(x_0, \ldots ,x}_k)\,]\,]'\,],
\end{multline*}
where $\mathrm{x_0, \ldots ,x}_k$ are the free number variables of
t, u, s. And by (c), $\mathrm{^*190_f^F}$, $\mathrm{^*78_f^F}$,
\begin{multline*}
\mathrm{(d)\;\;\vdash _1 E' \leftrightarrow \exists \delta
[\,\forall w A(t^w, u^w, s^w,\delta (w)) ~\&~ \exists \gamma
[\,\forall w
A(t_1^w, u_1^w, s_1^w,\gamma (w))}\\
\mathrm{ ~\&~ [\,E[\,\gamma (y_0, \ldots ,y}_l), \mathrm{\delta
(x_0, \ldots ,x}_k)\,]\,]'\,].
\end{multline*}
By the inductive hypothesis again, we have
\begin{multline*}
\mathrm{(e) \;\;\vdash _1 [\,E[\,rec(t_1, u_1, s_1), \delta (x_0,
\ldots ,x}_k)\,]\,]' \leftrightarrow
\\\mathrm{\exists \gamma [\,\forall w A(t_1^w, u_1^w, s_1^w,\gamma
(w))} \mathrm{ ~\&~ [\,E[\,\gamma (y_0, \ldots ,y}_l),
\mathrm{\delta (x_0, \ldots ,x}_k)\,]\,]'\,].
\end{multline*}
The left part of the $\leftrightarrow$ in (e) is just the result
of replacing the specified occurrence $\mathrm{rec(t, u, s)}$ by
an occurrence of $\mathrm{\delta (x_0, \ldots ,x}_k)$ in E. So, by
(d), (e) and the replacement theorem we get (a congruent of) (a).
\end{proof}

\begin{lemma}
Let $\,\mathrm{E \equiv E[\,rec(t, u, s)}\,]$ be any formula of
$\,\mathrm{\bf S_2}$ such that $\mathrm{rec(t, u, s)}$ is a
specified occurrence of a rec-plain term in $\mathrm{E}$ not in
the scope of any quantifier binding a free variable of
$\,\mathrm{rec(t, u, s)}$. Then, if $\,\mathrm{x_0, \ldots ,x}_k$
are the free number variables of $\mathrm{\,t, u, s}$ and the
conditions on $\gamma$, $\mathrm{w}$ are as in the definition of
$\,'\,$,
$$
\mathrm{\vdash _1 E' \leftrightarrow
\exists \gamma [\,\forall w  A(t^w, u^w, s^w,\gamma (w)) ~\&~
[\,E[\,\gamma (x_0, \ldots ,x}_k)\,]\,]'\,].
$$
\end{lemma}

\begin{proof}
The proof is by induction on the complexity of the formula E. For
E prime apply Lemma 4.7. For E composite, let r be the term
$\mathrm{\gamma (x_0, \ldots ,x}_k)$. Then $\mathrm{E[\,r\,]\,}$
will be of one of the forms $\mathrm{\neg B[\,r\,]},\;$ $\mathrm{
B[\,r\,] ~\&~ C},\;$ $\mathrm{B  ~\&~ C[\,r\,]},\;$ $\mathrm{
B[\,r\,] \vee C},\;$ $\mathrm{B \vee C[\,r\,]},\;$ $\mathrm{
B[\,r\,] \rightarrow C},\;$ $\mathrm{B \rightarrow C[\,r\,]},\;$
$\mathrm{ \forall x\, B(x)[\,r\,]},\;$ $\mathrm{ \exists x
\,B(x)[\,r\,]},\;$ $\mathrm{ \forall \alpha \,B(\alpha
)[\,r\,]},\;$ $\mathrm{ \exists \alpha \,B(\alpha )[\,r\,]}$.
Thanks to Lemma 4.4 and the conditions on the variables in the
present lemma and in the definition of $\,'$, we can apply for
each form the corresponding case of the functional version of
Lemma 25 of \cite{IM}. In the case that we treat here we will use
\begin{multline*}
\mathrm{^*189_f^F\;\;\exists ! \beta F(\beta ) \vdash \exists
\beta \,[\,F(\beta ) ~\&~ \forall \alpha \,D(\alpha , \beta )\,]
\leftrightarrow \forall \alpha \exists \beta \,[\, F(\beta ) ~\&~
D(\alpha , \beta )\,],}\\ \text{where}\;\alpha\;\text{does not
occur free in}\; \mathrm{F(\beta )}.
\end{multline*}
\textsc{Case} $\mathrm{E \equiv \forall \alpha \,B(\alpha) \equiv
\forall \alpha \,B(\alpha )\,[\,rec(t,u,s)\,]}$,\\ where
$\mathrm{rec(t, u, s)}$ is the occurrence to be eliminated, so
$\alpha$ does not occur free in t, u, s. From the definition of
$\,'\,$ with the inductive hypothesis
$$
\mathrm{\vdash _1 E'
\leftrightarrow \forall \alpha \exists \gamma \,[\,\forall w
A(t^w, u^w, s^w,\gamma (w)) ~\&~ [\,B(\alpha )\,[\,\gamma (x_0,
\ldots ,x}_k)\,]\,]'\,].
$$
Since $\alpha$ is not free in $\mathrm{\forall w A(t^w, u^w,
s^w,\gamma (w))}$, by $\mathrm{^*189_f^F}$ with Lemma 4.4
$$
\mathrm{\vdash _1 E' \leftrightarrow
\exists \gamma \,[\,\forall w
A(t^w, u^w, s^w,\gamma (w)) ~\&~ \forall \alpha
[\,B(\alpha )\,[\,\gamma (x_0, \ldots ,x}_k)\,]\,]'\,].
$$
\end{proof}

\begin{lemma}
Let $\mathrm{E}$ be a formula of $\,\mathrm{\bf S_2}$ in which the
rec-plain term $\mathrm{rec(t, u, s)}$ has some occurrences such
that no (number or function) variable free in t, u or s becomes
bound in $\mathrm{E}$ by a universal or existential quantifier.
Let $\mathrm{E[\,\gamma (x_0, \ldots , x}_k)\,]$ be the result of
replacing in $\mathrm{E}$ one or more specified occurrences of
$\,\mathrm{rec(t, u, s)}$ by $\mathrm{\gamma (x_0, \ldots ,
x}_k)$, where $\mathrm{x_0, \ldots ,x}_k$ are all the free number
variables of $\,\mathrm{t, u, s}$ and $\gamma$, $\mathrm{w}$ as in
the definition of $\,'\,$. Then
$$
\mathrm{\vdash _1 E' \leftrightarrow
\exists \gamma [\,\forall w  A(t^w, u^w, s^w,\gamma (w)) ~\&~
[\,E[\,\gamma (x_0, \ldots ,x}_k)\,]\,]'\,].
$$
\end{lemma}

\begin{proof}
The proof is by induction on the number $q$ of occurrences of
$\mathrm{rec(t, u, s)}$ that are replaced (by $\mathrm{\gamma
(x_0, \ldots , x}_k)$). For $q$=1 the lemma follows by Lemma 4.8.
For the inductive step, let $\mathrm{E \equiv E[\,rec(t, u,
s)}\,]$ indicate $q$ specified occurrences of $\mathrm{rec(t, u,
s)}$ in E, and let $\mathrm{E[\,\gamma (x_0, \ldots , x}_k)\,]$ be
the result of replacing these $q$ specified occurrences  by
$\mathrm{\gamma (x_0, \ldots , x}_k)$. By Lemma 4.8 we have
$$
\mathrm{\vdash _1 E' \leftrightarrow
\exists \gamma [\,\forall w  A(t^w, u^w, s^w,\gamma (w)) ~\&~
[\,E[\,\gamma (x_0, \ldots ,x}_k), \mathrm{rec(t, u,
s)}\,]\,]'\,],
$$
where $\mathrm{E[\,\gamma (x_0, \ldots ,x}_k),
\mathrm{rec(t, u, s)}\,]\,$ is the result of replacing in E the
first of the $q$ specified occurrences of $\mathrm{rec(t, u, s)}$
by $\mathrm{\gamma (x_0, \ldots , x}_k)$. By the inductive
hypothesis with the replacement theorem,
\begin{multline*}
\mathrm{\vdash _1 E' \leftrightarrow \exists \gamma [\,\forall w
A(t^w, u^w, s^w,\gamma (w)) ~\&~ \exists \delta [\,\forall w
A(t^w, u^w, s^w,\delta (w))}\\
\mathrm{ ~\&~ [\,E[\,\gamma (x_0, \ldots ,x}_k), \mathrm{\delta
(x_0, \ldots ,x}_k)\,]\,]'\,],
\end{multline*}
where $\mathrm{ E[\,\gamma (x_0, \ldots ,x}_k), \mathrm{\delta
(x_0, \ldots ,x}_k)\,]\,$ is the result of replacing the other
$q-1$ specified occurrences of $\mathrm{rec(t, u, s)}$ in
$\mathrm{ E[\,\gamma (x_0, \ldots , x}_k), \mathrm{rec(t, u, s)}
\,]$    by $\mathrm{\delta (x_0, \ldots ,x}_k)$. By  Lemma 4.4 and
the following case of the functional version of Lemma 25,
\cite{IM},
$$
\mathrm{^*183^F\;\;\exists ! \beta \,F(\beta )
\vdash \exists  \alpha\,[\,F(\alpha ) ~\&~ C(\alpha , \alpha )\,]
\leftrightarrow
\exists \alpha \,[\, F(\alpha) ~\&~  \exists \beta \, [\,F(\beta )
~\&~ C(\alpha , \beta )\,]\,]\,,}
$$
where  $\alpha$ does not occur  free in  $\mathrm{F(\beta )}$ and
is free for $\beta$ in $\mathrm{F(\beta )}$ and in
$\mathrm{C(\alpha , \beta )}$,
$$
\mathrm{\vdash _1 E' \leftrightarrow
\exists \gamma [\,\forall w  A(t^w, u^w, s^w,\gamma (w)) ~\&~
[\,E[\,\gamma (x_0, \ldots ,x}_k), \mathrm{\gamma (x_0, \ldots
,x}_k)\,]\,]'\,],
$$
and, since $\mathrm{E[\,\gamma (x_0, \ldots
,x}_k), \mathrm{\gamma (x_0, \ldots ,x}_k)\,]$ is just
$\mathrm{E[\,\gamma (x_0, \ldots ,x}_k)\,]$, we get the result.
\end{proof}
In the next lemma we will use the following case from the
functional version of Lemma 25, \cite{IM}:
$$
\mathrm{^*182^F\;\;\exists
! \beta F(\beta ), \forall \beta \,C(\beta ) \vdash \exists  \beta
\,[\,F(\beta ) ~\&~ C(\beta )\,]\,}.
$$

\begin{lemma}
If $\,\mathrm{E}$ is any axiom of $\,\mathrm{\bf S_2}$, then
$\mathrm{\vdash _1 E'}$.
\end{lemma}

\begin{proof}
(i) If E \emph{is an axiom of $\mathrm{\bf S_2}$ by an axiom
schema of the propositional logic}, then E$'$ is equivalent in
$\mathrm{\bf S_1}$ to  an axiom of $\mathrm{\bf S_1}$ by the same
axiom schema. This follows by the fact that, by its definition,
the translation $'$ preserves the logical operators, together with
the fact that the translations of different instances of the same
formula may differ only in their bound variables, so they are
congruent, hence equivalent. So by the replacement theorem for
equivalence, $\mathrm{\vdash _1 E'}$.

\vskip 0.1cm

(ii) \emph{The logical axioms for the quantifiers} need a
different treatment. We give the proof for axiom schema 10N, and
the other cases follow similarly.

\emph{Axiom schema} 10N: $\mathrm{E \equiv \forall x \,B(x)
\rightarrow B(r)}$, where r is a term of $\mathrm{\bf S_2}$ free
for $\mathrm{x}$ in $\mathrm{B(x)}$. Then $\mathrm{E' \equiv
\forall x \,[\,B(x)\,]' \rightarrow [\,B(r)\,]'}$.

Ia. If x has no free occurrences in B(x), or r is rec-less and x
does not occur free in any rec-occurrence, we simply choose the
same bound variables at corresponding steps in the elimination
processes in B(x) and B(r), so the resulting formula
$\mathrm{\forall x\, C(x) \rightarrow C(r)}$ is congruent to E$'$
and is an axiom of $\mathrm{\bf S_1}$ by 10N, so $\mathrm{\vdash
_1 E'}$.

Ib. If x occurs free in some of the rec-occurrences of B(x) and r
is rec-less, we show first, by induction on the number $g$ of
logical operators in B(x),
$$
\mathrm{(A)\;\;\vdash _1 [\,B(x)\,]'(x/r)
\leftrightarrow [\,B(r)\,]'}.
$$
The case of B(x) prime is proved by induction on the number $q$ of
rec-occurrences in B(x) as follows. For $q>0$, consider the first
occurrence of a rec-plain term in B(x). If x does not occur free
in it, then (A) follows from the inductive hypothesis and the
definition of $\,'\,$. Otherwise, let $\mathrm{rec(t(x), u(x),
s(x))}$  with x occurring free in t(x), u(x) or s(x) be the
considered occurrence, so $\mathrm{B(x) \equiv B(x)[rec(t(x),
u(x), s(x))]}$, and let $\mathrm{x_0, \ldots , x_{\textit{k}},x}$
be all the free number variables of t(x), u(x), s(x). Then, by the
definition of $\,'\,$ we get
$$
\mathrm{(a)\,\vdash _1 [\,B(x)\,]' \leftrightarrow \exists
\gamma [\,\forall w
A(t(x)^w, u(x)^w, s(x)^w, \gamma (w))
 ~\&~ [\,B(x)\,[\,\gamma (x_0, \ldots , x_{\textit{k}},
x)]\,]'\,]},
$$
where w, $\gamma$ are new variables and all the bound variables
are chosen so that r is free for x in the displayed formula.

For simplicity, let x, z be the free number variables of r, with z
not included in $\mathrm{x_0, \ldots , x_{\textit{k}},x}$, so that
$\mathrm{r \equiv r(x, z)}$. By eliminating the corresponding
rec-occurrence from B(r) by the definition of $\,'\,$ we get
$$
\mathrm{(b)\;\vdash _1\,[\,B(r)\,]' \leftrightarrow \exists
\delta \,[\,\forall w
A(t(r)^w, u(r)^w, s(r)^w, \delta (w))
 ~\&~ [\,B(r)\,[\,\delta (x_0, \ldots ,
x_{\textit{k}},x,z)\,]\,]']}.
$$
From (a) by $\forall$-introduction and elimination, since  w is
the only number variable free in the formula
 $\mathrm{A(t(x)^w, u(x)^w, s(x)^w, \gamma (w))}$, we have
\begin{multline*}
\mathrm{(a1)\;\;\vdash _1 [\,B(x)\,]'(x/r) \leftrightarrow \exists
\gamma
[\forall w A(t(x)^w, u(x)^w, s(x)^w, \gamma (w))}\\
 \mathrm{~\&~ [\,B(x)\,[\,\gamma (x_0, \ldots , x_{\textit{k}},
x)\,]\,]'(x/r)]}.
\end{multline*}
Assume $\mathrm{(a2)\,[\,B(x)\,]'(x/r)}.$ We will get
$\mathrm{[\,B(r)\,]'}$. By (a1) and (a2) we may assume
$$
\mathrm{(a3)\;\;\forall w
A(t(x)^w, u(x)^w, s(x)^w, \gamma (w))}\\
\mathrm{ ~\&~ [\,B(x)\,[\,\gamma (x_0, \ldots , x_{\textit{k}},
x)\,]\,]'(x/r)\,}.
$$
By the inductive hypothesis,
$$
\mathrm{(a4)\;\;\vdash_1
[\,B(x)\,[\,\gamma (x_0, \ldots , x_{\textit{k}},
x)]\,]'(x/r) \leftrightarrow
[\,B(r)\,[\,\gamma (x_0, \ldots , x_{\textit{k}},
r(x,z))\,]\,]'}.
$$
By Lemma 4.4 we may assume $ \mathrm{(a5)\;\;\forall w \,A(t(r)^w,
u(r)^w, s(r)^w, \delta (w))}. $ Now, specializing for $\mathrm{q_0
\equiv \langle x_0, \ldots ,x_{\textit{k}}, r(x,z)\rangle }$ and
$\mathrm{q_1 \equiv \langle x_0, \ldots ,x_{\textit{k}},x,z\rangle
}$ from the first conjunct of (a3) and from (a5) respectively,
since $\mathrm{t(x)^{q_0}}$ is just $\mathrm{t(r)^{q_1}}$, and
similarly for u(x), s(x), using Lemma 4.2 ($\mathrm{\vdash _1
\forall y \exists ! m \, A(x, \alpha , y, m)}$) and then
$\forall$-introds. we get
$$
\mathrm{(c)\;\;\forall x_0 \ldots
\forall x_{\textit{k}} \forall x \forall z\,\gamma
(x_0, \ldots , x_{\textit{k}}, r(x,z))= \delta (x_0, \ldots ,
x_{\textit{k}},x, z)}.
$$

We will use now the following

\textsc{Fact}. For any formula $\mathrm{D[x]}$ of {\bf S$_2$} with
a specified occurrence of x indicated, if s, t are rec-less terms,
$\mathrm{x_0, \ldots ,x_{\textit{k}}}$ include all the free number
variables of s and t, and no free function variable of s or t
becomes bound in the corresponding replacements, then
$$
\mathrm{\forall x_0 \ldots
\forall x_{\textit{k}}\; s = t \vdash _1
 [D[s]]' \leftrightarrow [D[t]]'}.
$$
The proof is by induction on the number of logical operators in
$\mathrm{D[x]}$. The basis (case of prime formulas) is obtained by
induction on the number $q$ of rec-occurrences in the formula, by
use of the replacement theorem for $\mathrm{\bf S_1}$.

\vskip 0.1cm

By the above fact and (c), we get now
$$
\mathrm{(d)\;\;\
[\,B(r)\,[\,\delta (x_0, \ldots , x_{\textit{k}},x,z)]\,]'
\leftrightarrow [\,B(r)\,[\,\gamma (x_0, \ldots ,
x_{\textit{k}},r(x,z))]\,]'}.
$$
From the second conjunct of (a3) with (a4) and with (d) we get
$$
\mathrm{(e)\;\;\
[\,B(r)\,[\,\delta (x_0, \ldots , x_{\textit{k}},x,z)]\,]'}.
$$
Now from (a5) and (e) with $\exists \delta$-introduction,
 discharging (a3) and (a5), we get the
right part of (b) and finally $\mathrm{[\,B(r)\,]'}$.

The other direction of the equivalence of (A) is obtained
similarly and the case of composite formulas (inductive step for
$g>0$) follows easily. From (A) follows immediately that
$\mathrm{E' \equiv \forall x \,[\,B(x)\,]' \rightarrow
[\,B(r)\,]'}$ is a congruent of an axiom of $\mathrm{\bf S_1}$ by
the same axiom schema, so $\mathrm{\vdash _1E'}$.

\vskip 0.1cm

II. If r has some rec-occurrences, the result is obtained by an
induction on the number $q$ of these occurrences. If r contains
$q$ ($q>0$) rec-occurrences, we consider the first rec-plain
occurrence in r, say rec(t, u, s), so
$$
\mathrm{(a) \;\; E \equiv E (r\,[\,rec(t, u, s)\,]\,) \equiv
\forall x\, B(x) \rightarrow B(r\,[\,rec(t, u, s)\,])}.
$$
By Lemma 4.9,
$$
\mathrm{(b)\;\;\vdash _1 E' \leftrightarrow
\exists \gamma \,[\,\forall w  A(t^w, u^w, s^w,\gamma (w)) ~\&~
[\,E(r\,[\,\gamma (x_0, \ldots ,x_{\textit{k}})\,]\,)\,]'\,]},
$$
with $\gamma$ new for E and $\mathrm{x_0, \ldots
,x_{\textit{k}}}$,w as usual. By the inductive hypothesis,
$$
\mathrm{(c) \;\; \vdash _1
[\,\forall x\, B(x) \rightarrow B(r\,[\,\gamma (x_0, \ldots
,x_{\textit{k}})\,])\,]'}.
$$
But $\mathrm{E(r\,[\,\gamma (x_0, \ldots ,x_{\textit{k}})\,]\,)}$
is just $\mathrm{ \forall x \,B(x) \rightarrow B(r\,[\,\gamma
(x_0, \ldots ,x_{\textit{k}})\,])}$, so by Lemma 4.4 with
$\mathrm{^*182^F}$, from (b) and the universal closure of (c) we
get $\mathrm{\vdash _1 E'}$.

\vskip 0.1cm

(iii) \emph{The axiom schema of induction}. $ \mathrm{E \equiv
A(0) ~\& ~ \forall x( A(x) \rightarrow A(x') ) \rightarrow A(x)}.$
Then $\mathrm{E' \equiv [\,A(0)\,]' ~\& ~ \forall x( [\,A(x)\,]'
 \rightarrow
[\,A(x')\,]')  \rightarrow [\,A(x)\,]'} $, and the result follows
by the arguments of (ii), Ia and Ib.

\vskip 0.1cm

(iv) \emph{The axiom schema of $\lambda$-conversion}. $\mathrm{E
\equiv (\lambda x.r(x))(p) = r(p)}$, where r(x), p are terms of
$\mathrm{\bf S_2}$ and p is free for x in r(x).

If r(x), p are rec-less, then E$'$ is E and is an axiom of
$\mathrm{\bf S_1}$ by the same schema.

If there are $q$ ($q>0$) rec-occurrences in total in r(x) and p,
we consider the first occurrence of a rec-plain term in r(x) (in
case that r(x) is rec-less, then consider the first such in p),
say rec(t, u, s), with free number variables the $\mathrm{x_0,
\ldots , x}_k$, and $\gamma$, w new. Then
$$
\mathrm{ E \equiv E [\,rec(t, u, s)\,] \equiv
(\lambda x.r(x)[\,rec(t, u, s)\,])(p) \!=\! (r(x)[\,rec(t, u,
s)\,])(x/p)}.
$$
By Lemma 4.9,
\begin{multline*}
\mathrm{(a)\;\;\vdash _1 E' \leftrightarrow \exists \gamma
[\,\forall w  A(t^w, u^w, s^w,\gamma (w)) ~\&~}\\
\mathrm{ [\,(\lambda x.r(x)\,[\,\gamma (x_0, \ldots
,x_{\textit{k}})\,])(p) = (r(x)\,[\,\gamma (x_0, \ldots
,x_{\textit{k}})\,])(x/p)\,]'\,]}.
\end{multline*}
The instance of $\lambda$-conversion shown in (a) (x may or may
not be one of $\mathrm{x_0, \ldots , x}_k$, but in both cases p is
free for x in $\mathrm{r(x)\,[\,\gamma (x_0, \ldots
,x_{\textit{k}})\,]}$) has $q-1$ rec-occurrences in the terms
involved, so the induction hypothesis applies and we obtain
$$
\mathrm{(b)\;\;
\vdash _1[\,(\lambda x.r(x)\,[\,\gamma (x_0, \ldots
,x_{\textit{k}})\,])(p) = (r(x)\,[\,\gamma (x_0, \ldots
,x_{\textit{k}})\,])(x/p)\,]'}.
$$
By Lemma 4.4 with $\mathrm{^*182^F}$, we get now from (a) and (b)
that $\mathrm{\vdash _1E'}$. In case that the rec-occurrence to be
eliminated is in the term p, the argument is similar.

\vskip 0.1cm

(v) \emph{The axiom schema} $\mathrm{AC_{00}!}$.
$$
\mathrm{E \equiv \forall x \exists y \,[\, A(x, y) ~\& ~
\forall z( A(x, z) \rightarrow y = z )\,]
 \rightarrow \exists \alpha \forall x A(x, \alpha (x))}.
$$
By the arguments of (ii), Ia and Ib, E$'$ is (congruent to) an
axiom of $\mathrm{\bf S_1}$ by the same axiom schema, so
$\mathrm{\vdash _1E'}$.

\vskip 0.1cm

(vi) \emph{The axiom $\mathrm{Rec \;\;A(x, \alpha ,y, rec(x,
\alpha ,y) )}$}.
$$
\mathrm{E \equiv \exists \beta \,[\,\beta (0)=x ~\&~ \forall z\,
\beta (z')=\alpha (\langle \beta (z),z\rangle ) ~\&~ \beta
(y)=rec(x, \alpha ,y) \,]\,,}
$$
and so we have
$$
\mathrm{E' \equiv\exists \beta \,[\,\beta (0)=x
~\&~ \forall z\, \beta (z')=\alpha (\langle \beta (z),z\rangle )
~\&~ [\,\beta
(y)=rec(x, \alpha ,y) \,]'\,]\,}.
$$
We have already shown
$$
\mathrm{(a)\;\;\vdash _1[\,\beta (y)=rec(x, \alpha ,y) \,]'
\leftrightarrow \exists \gamma \,[\,\forall w A(x^w, \alpha ^w, y^w,
\gamma (w)) ~\&~\beta (y) = \gamma (x, y)\,]\,.}
$$
By Lemma 4.1 we may assume
\begin{multline*}
\mathrm{(b)\; \;\beta (0) = x ~\&~ \forall z \beta (z')=\alpha
(\langle \beta (z),z \rangle )}\\
\mathrm{ ~\&~\forall \delta ( \delta (0) = x ~\&~ \forall z \delta
(z')=\alpha (\langle \delta (z),z \rangle ) \rightarrow \beta =
\delta )}.
\end{multline*}
By Lemma 4.4 we may assume $\mathrm{(c)\;\;\forall w A(x^w, \alpha
^w, y^w, \gamma (w))}$, and from this by specializing for
$\mathrm{\langle x, y \rangle}$ we may also assume
$$
\mathrm{(c1)\;\; \delta (0) = x ~\&~ \forall
z \,\delta (z')=\alpha (\langle \delta (z),z \rangle )
 ~ \&~ \delta
(y) = \gamma (x, y).}
$$
By (b) and (c1) we get $\beta = \delta$, and then also
$\mathrm{(d)\;\;\beta (y) = \gamma (x, y)}$. After $\exists
\delta$-elimination discharging (c1), we get from (c), (d),
$$
\mathrm{(e)\;\;\forall w A(x^w,
\alpha ^w, y^w, \gamma (w)) ~\&~ \beta (y) = \gamma (x, y)}.
$$
And after $\exists \gamma$-introduction, $\exists
\gamma$-elimination discharging (c), from (e), (b) we get
\begin{multline*}
\mathrm{(f)\;\; \beta (0) = x ~\&~ \forall
z \,\beta (z')=\alpha (\langle \beta (z),z \rangle )}\\
~\&~ \mathrm{\exists \gamma \,[\,\forall w A(x^w, \alpha ^w, y^w,
\gamma (w)) ~\&~ \beta (y) = \gamma (x, y)\,]\,}.
\end{multline*}
After $\exists\beta$-introduction, $\exists\beta$-elimination
discharging (b), we get E$'$ in $\mathrm{\bf S_1}$ (using (a) and
the replacement theorem for equivalence).

\vskip 0.1cm

(vii) The remaining axioms are finitely many axioms not containing
the constant rec, and they are also axioms of $\mathrm{\bf S_1}$.
\end{proof}

\begin{lemma}
If $\:\mathrm{E}$ is an immediate consequence of $\:\mathrm{F}$
 ($\:\mathrm{F}$
and $\:\mathrm{G}$) in $\mathrm{\bf S_2}$, then $\mathrm{E'}$ is
an immediate consequence of $\:\mathrm{F'}$
 ($\:\mathrm{F'}$ and $\:\mathrm{G'}$) in $\mathrm{\bf S_1}$.
\end{lemma}

\begin{proof}
Since by the definition of the translation $\,'\,$ the logical
operators are preserved and since, by Lemma 4.6, no  free
variables are introduced or removed, and also since congruent
formulas are equivalent, it follows that to each instance of a
rule of $\mathrm{\bf S_2}$ corresponds  an instance of the same
rule in $\mathrm{\bf S_1}$.
\end{proof}

We conclude that if $\,\mathrm{\Gamma \vdash _2 E}$, where
$\Gamma$ is a list of formulas and E is a formula of $\mathrm{\bf
S_2}$ and no variables are varied in the deduction, then
$\mathrm{\Gamma '\vdash _1 E'}$. So elimination relation III is
satisfied.

\section{Comparison of {\bf M} and {\bf EL}}

\subsection{ Introduction of the other function(al) constants}

\subsubsection{}

Having the recursor constant rec in the formalism, it is immediate
that any constant for a function with a primitive recursive
description in which are used only functions with names already in
the symbolism can be added definitionally. The needed translation
is trivial, it amounts just to the replacement of each occurrence
of the new constant by the corresponding (longer) term provided by
the formalism. More concretely, constants for all the primitive
recursive functions can be added definitionally, successively
according to their primitive recursive descriptions, as follows:

\noindent\boldmath $\cdot$ \unboldmath For the initial functions
and for functions defined by composition from functions for which
we already have constants, it is very easy to find terms
expressing them.

\noindent\boldmath $\cdot$ \unboldmath For the case of definition
by primitive recursion we use the constant rec: for example if
$f(x,0)=g(x)$ and $f(x, y+1)=h(f(y),y,x)$, we introduce a new
constant f$_j$ by $\mathrm{f_{\textit{j}}(x,y)=rec(g(x), \lambda
z. h((z)_0,(z)_1,x),y)}$, if we have already in our symbolism
constants g and h for $g$ and $h$, respectively.

We note that in this way, not only functions, but also functionals
can be added in a formalism having a recursor, like  {\bf EL} or
{\bf M} + Rec. We also note that the equality axioms for the new
constants become provable.

\subsection{Comparison of {\bf M} and {\bf EL}}

\subsubsection{}

Let $\mathrm{{\bf M^+}}$ be obtained by adding to {\bf M} + Rec
all the (infinitely many) function constants of {\bf HA}, with
their defining axioms, extending also all axiom schemata to the
new language.

Let $\mathrm{{\bf EL^+}}$ be obtained by adding to {\bf EL} all
the (finitely many) functional constants of {\bf M}, with their
defining axioms, extending also all axiom schemata to the new
language.

We see that the languages of the extended systems $\mathrm{{\bf
M^+}}$ and $\mathrm{{\bf EL^+}}$ coincide (with trivial
differences). Using the relations between the function existence
principles that we have obtained as well as the equivalence of the
different definitions of the recursor constant, we arrive at the
following: \vskip 0.1cm

 {\bf EL$^+$}+ $\mathrm{CF\!_d}$ = {\bf HA$_1$} +
fin.list({\bf M}) + QF-AC$_{00}$ + CF$\!_\mathrm{d}$ =

\hskip 1.5cm {\bf IA$_1$} + Rec + inf.list({\bf HA}) + AC$_{00}$!
= {\bf M$^+$}.

\begin{theorem}
$\mathrm{{\bf EL^+}\!\!+ CF\!_d}$ is a conservative (in fact
definitional) extension of $\,\mathrm{{\bf M}}$.
\end{theorem}

\begin{proof}
It suffices to observe that every proof in $\mathrm{{\bf EL^+
}\!\!+ CF\!_d}$ is done in a finite subsystem of it, so in a
definitional extension of $\mathrm{{\bf M}}$.
\end{proof}

\begin{theorem}
$\mathrm{{\bf M^+}}$ is a conservative (in fact definitional)
extension of $\,\mathrm{{\bf EL }+ CF\!_d}$.
\end{theorem}

\begin{corollary}
The systems $\,\mathrm{{\bf M^+}}$ and  $\,\mathrm{{\bf EL^+
}\!\!+ CF\!_d}$ essentially coincide, so $\,\mathrm{{\bf M}}$ and
$\,\mathrm{{\bf EL }+ CF\!_d}$ are essentially equivalent, in the
sense that they have a common conservative extension obtained by
definitional extensions.
\end{corollary}

\vskip 0.1cm

In \cite{Troelstra1974} p. 585, a result of N. Goodman is
mentioned and used, stating that $\mathrm{{\bf EL_1}}$ is
conservative over {\bf HA}, where $\mathrm{{\bf EL_1}}$ is {\bf
EL} + AC$_{01}$, where AC$_{01}$ (which entails AC$_{00}$!) is the
countable choice schema assumed by {\bf FIM}. With our previous
results, we obtain the following.

\begin{theorem}
 $\mathrm{{\bf EL} + CF\!_d}$ is a conservative extension of
 $\,\mathrm{{\bf HA}}$.
\end{theorem}

\begin{theorem}
 $\mathrm{{\bf M^+}}$ is a conservative extension of
 $\,\mathrm{{\bf HA}}$.
\end{theorem}

\begin{theorem}
 $\mathrm{{\bf M}}$ is a conservative
extension of $\,\mathrm{\bf IA_0}$.
\end{theorem}

\section{Elimination of the symbol $\lambda$ from {\bf EL}
and $\mathrm{{\bf IA_1} + QF\text{-}AC_{00}}$}

\subsection{}

From \cite{JRMPhD}  it is known that $\lambda$ can be eliminated
from the formal systems that S. C. Kleene set up to formalize
parts of intuitionistic analysis, including {\bf M}. The proof
uses $\mathrm{AC_{00}!}$ so, by the previous results, it is not
valid for the cases of {\bf EL} and $\mathrm{{\bf IA_1} +
QF\text{-}AC_{00}}$.\footnote{Motivated by Iris Loeb's question
whether this method of elimination (that we initially applied) for
{\bf EL} would have worked for {\bf M}, we considered also the
case of {\bf IA$_1$} + QF-AC$_{00}$.} Here we modify a part of it,
and obtain the corresponding results.

Let $\mathrm{\bf EL - \lambda}$ be obtained from {\bf EL} by
omitting the symbol $\lambda$, the corresponding functor formation
rule and the axiom schema of $\lambda$-conversion. As axioms for
the constant rec we include in both systems the following version,
equivalent by logic to the one used in the definition of {\bf EL},
with terms and a functor in the places of the number variables and
the function variable, respectively:
$$
\mathrm{REC}\;\;\left\{
\begin{array}{ll}
    \mathrm{rec (x, \alpha , 0) = x,} \\
    \mathrm{rec (x, \alpha ,S(y))=
    \alpha (\langle rec(x, \alpha , y), y \rangle ).}
\end{array}\right.
$$
Both systems include $\mathrm{QF\text{-}AC_{00}}$. Instead of it
we choose to include, in both systems, the following term-version
of it:
$$
\mathrm{QF_t\text{-}AC_{00}\;\;\;
\forall x \exists y\,
t(\langle x, y \rangle ) = 0
 \rightarrow \exists \alpha \forall x \,
t(\langle x, \alpha (x) \rangle ) = 0},
$$
where x is free for y in $\mathrm{t(\langle x, y \rangle )}$ and
$\alpha$ does not occur in $\mathrm{t(\langle x, y \rangle )}$.
Similarly, we consider $\mathrm{{\bf IA_1} + QF\text{-}AC_{00}}$
(its version with $\mathrm{QF_t\text{-}AC_{00}}$) and the
corresponding $\mathrm{{\bf IA_1} + QF\text{-}AC_{00} - \lambda}$.

In the following $\mathrm{\bf S_2}$ is {\bf EL} or $\mathrm{{\bf
IA_1} + QF\text{-}AC_{00}}$ and  $\mathrm{\bf S_1}$ is
$\mathrm{\bf S_2 - \lambda}$. By $\vdash _1$, $\vdash _2$ we
denote provability in $\mathrm{\bf S_1}$, $\mathrm{\bf S_2}$,
respectively.

\vskip 0.1cm

\textsc{Proposition}. \textit{Over $\mathrm{\bf HA_1}$,
$\mathrm{\bf IA_1}$ and the corresponding systems without
$\lambda$, $\mathrm{QF\text{-}AC_{00}}$ is interderivable with
$\mathrm{QF _t\text{-}AC_{00}}$.}

\begin{proof}
Observe that $\mathrm{QF _t\text{-}AC_{00}}$ is a special case of
$\mathrm{QF\text{-}AC_{00}}$. We obtain
$\mathrm{QF\text{-}AC_{00}}$ from $\mathrm{QF _t\text{-}AC_{00}}$
as follows. In all the systems mentioned, for any quantifier-free
formula $\mathrm{A(x, y)}$, we can find a term $\mathrm{q(x, y)}$
with the same free variables as $\mathrm{A(x, y)}$, such that
$\mathrm{\vdash A(x,y) \leftrightarrow q(x,y) = 0}$. We get easily
the result if we consider the term $\mathrm{t(w) \equiv q((w)_0,
(w)_1)}$, for which we obtain $\mathrm{\vdash A(x, y)
\leftrightarrow t(\langle x, y \rangle ) = 0}$.
\end{proof}

\subsection{}

To obtain $\lambda$-eliminability in \cite{JRMPhD} a translation
$\,'\,$ is defined and then it is shown that the elimination
relations are satisfied. We give the definition for prime
formulas; for composite, the translation $\,'\,$ is defined so
that the logical operations are preserved.

\textsc{Remark} \emph{on terminology and notation.} We use
terminology and notation similar to those used in the case of rec.
In particular, if we consider a specified occurrence of a term or
functor R in an expression E, we write $\mathrm{E\,[\,R\,]\,}$ to
indicate this occurrence, and when this occurrence is replaced by
a function variable, say $\alpha$, we write $\mathrm{E\,[\,\alpha
\,]\,}$ to denote the result of this replacement.

\textsc{Definition.} If P is any prime formula of $\mathrm{\bf
S_2}$ with no $\lambda$'s, then P$'$ is P. Otherwise, if
$\mathrm{\lambda x.s(x)}$ is the first (free) $\lambda$-occurrence
in P, in which case we use the notation $\mathrm{P\,[\,\lambda
x.s(x)\,]}$, then
$$
\mathrm{P' \equiv \exists \alpha\,[\, \forall x\, [\,s(x) =
\alpha (x)\,]' ~\&~ [\,P\,[\,\alpha \,]\,]'\,]},
$$
where $\alpha$ is a function variable which does not occur in P,
and $\mathrm{P\,[\,\alpha \,]}$ is obtained from P by replacing
the occurrence $\mathrm{\lambda x.s(x)}$ by an occurrence of
$\alpha$.

The only point of the proof in \cite{JRMPhD} where AC$_{00}$!
 is used (except the treatment of AC$_{00}$! itself) is in order to
obtain $ \;\mathrm{\vdash _1 \exists ! \alpha\, \forall x\,
[\,t(x) = \alpha (x)\,]' }$. Here we present some lemmas by which
we obtain this result using only   $\mathrm{QF\text{-}AC_{00}}$,
so, after treating the axioms not included in the systems of
\cite{JRMPhD}, we obtain $\lambda$-eliminability for {\bf EL} and
$\mathrm{{\bf IA_1} + QF\text{-}AC_{00}}$.

\vskip 0.1cm

\textsc{Remark}. Free substitution of a free (number) variable by
a $\lambda$-less term commutes with the translation $\,'\,$ (it
follows by induction on the number of logical operators of the
considered formula, where the case of prime formulas follows by
induction on the number of $\lambda$-occurrences). From this we
also obtain:

(a) Consider $\mathrm{E \equiv \forall x A(x) \rightarrow A(t)}$,
where A(x) is any formula of $\mathrm{\bf S_2}$ (with t free for x
in A(x)), an axiom of $\mathrm{\bf S_2}$ by schema 10N. If t has
no $\lambda$'s, then $ \mathrm{\vdash_1E'}$.

(b) If s, t are $\lambda$-less terms free for z in a formula B(z)
of $\mathrm{\bf S_2}$, then
$$
 \mathrm{s=t \vdash_1[B(s)]'
\leftrightarrow [B(t)]'}.
$$
Corresponding results are obtained for function variables and
$\lambda$-less functors (and in particular for the $\forall
\alpha$-elimination schema 10F).

\begin{lemma}
Let $\mathrm{r(x)}$ be any term of $\,\mathrm{\bf S_2}$ with
$\alpha$ not free in it. Then
$$
\mathrm{\vdash _1 \exists \alpha\, \forall x\, [\,r(x) =
\alpha (x)\,]' }.
$$
\end{lemma}

\begin{proof}
By induction on the number $q$ of $\lambda$'s in the term r(x).

For $q = 0$, since $\mathrm{\;\vdash _1 \forall x \exists w\,r(x)
= w }$, just apply $\mathrm{QF\text{-}AC_{00}}$.

Let r(x) have $q>0$ $\lambda$'s, and let $\mathrm{\lambda z. s(z,
x)}$ be the first $\lambda$-occurrence in r(x), so that
$\mathrm{r(x) \equiv r(x)\,[\,\lambda z. s(z, x)\,]}$. We have to
show
$$
\mathrm{(a) \;\vdash_1 \exists \alpha
\forall x \exists \beta \,[\,\forall z \,[\,s(z, x) = \beta (z)
\,]' ~\&~ \,[\,r(x)[\,\beta \,] = \alpha (x)\,]'\,].}
$$
We may assume $ \mathrm{(b) \; \forall z \,[\,s(z, x) = \beta (z)
\,]'}$ and $\mathrm{(c) \;\forall x \,[\,r(x)[\,\beta \,] = \alpha
(x)\,]'} $, since we have both $\mathrm{(d) \;\vdash _1 \exists
\beta \forall z \,[\,s(z, x) = \beta (z) \,]'}$ and $\mathrm{(e)
\;\vdash _1 \exists \alpha \forall x \,[\,r(x)[\,\beta \,] =
\alpha (x)\,]'}$ by the inductive hypothesis. From (b) and (c)
with $\forall$x-elim., $\&$-introd. and $\exists \beta$-introd. we
get
$$
\mathrm{(f) \;\;\exists \beta \,[\,\forall z \,[\,s(z, x) = \beta
(z) \,]' ~\&~ [\,r(x)[\,\beta \,] = \alpha (x)\,]'\,]}.
$$
and with $\exists \beta$-elim. disch. (b) (from (d)), and then
$\forall$x-introd. (x is not free in (c)) and
$\exists\alpha$-introd., $\exists\alpha$-elim. disch. (c) (from
(e)), we get (a).
\end{proof}

\begin{lemma}
Let $\mathrm{t}$ be any term of $\mathrm{\,\bf S_2}$ with
$\mathrm{z}$ not free in it. Then
$$
\mathrm{\vdash _1 \exists ! z\,[\,t = z\,]'}.
$$
\end{lemma}

\begin{proof}
By induction on the number $q$ of $\lambda$'s in t.

For $q = 0$ the lemma follows by $\mathrm{\vdash _1 \exists ! z\;t
= z }$.

If t has $q>0$ $\lambda$'s and $\mathrm{\lambda x. s(x)}$ is the
first $\lambda$-occurrence in t, so that $\mathrm{t \equiv
t\,[\,\lambda x. s(x)\,]\,}$, we have $\mathrm{\,[\,t = z\,]'
\equiv \exists \alpha \,[\,\forall x \,[\,s(x) = \alpha (x) \,]'
~\&~ \,[\,t[\,\alpha\,] = z\,]'\,]\,.}$ We may assume
$$
\mathrm{(a)\;\;\forall x \,[\,s(x) = \alpha (x)
\,]'\;\;\text{and}\;\; (b)\;\,[\,t[\,\alpha\,] = z\,]',}
$$
since by Lemma 6.1, we have $\mathrm{(c)\;\vdash _1 \exists \alpha
\forall x \,[\,s(x) = \alpha (x) \,]'}$ and by the inductive
hypothesis $\mathrm{(d)\;\vdash _1 \exists ! z \,[\,t[\,\alpha\,]
= z\,]'.}$ By (a), (b) with $\exists\alpha$-introd. we get
$\mathrm{(e)\;[\,t = z\,]'.}$ For the uniqueness, we assume
$$
\mathrm{(f)\;\;
[\,t = y\,]'\, \equiv \exists \beta \,[\,\forall x \,[\,s(x) =
\beta (x) \,]' ~\&~ \,[\,t[\,\beta\,] = y\,]'\,]\,},
$$
so we may also assume
$$
\mathrm{(g1)\;\;
\forall x \,[\,s(x) = \beta (x) \,]'\;\;\text{and}\;\;
(g2)\;\; \,[\,t[\,\beta\,] = y\,]'.}
$$
From (a) and (g1) we get $\alpha = \beta$ as follows: from (a) and
(g1), by $\forall$-elims, we get
$$
\mathrm{(h1)\;\;[\,s(x) = \alpha (x) \,]'\;\;\text{
and} \;\;(h2)\;\;[\,s(x) = \beta (x) \,]',}
$$
respectively. By the inductive hypothesis, $\mathrm{\vdash _1
\exists ! z\,[\,s(x) = z\,]'},$ so by (h1), (h2) and the
\textsc{Remark} we get $\mathrm{\alpha (x) = \beta (x)}$, so
$\mathrm{\forall x \,\alpha (x) = \beta (x) }$ (x is not free in
(a), (g1)), so $\mathrm{(i)\; \alpha = \beta}.$ By the
\textsc{Remark}, from (g2), (i), we get $\mathrm{(j)\;
\,[\,t[\,\alpha\,] = y\,]'}$, and then from (b), (j) and (d) we
get y = z. So after $\exists\beta$-elim. disch. (g1), (g2), with
$\rightarrow$-introd. disch. (f), and with $\forall$y-introd. and
(e), we get
$$
\mathrm{[\,t = z\,]' ~\&~ \forall y ([\,t = y\,]' \rightarrow y =
z)},
$$
and with $\exists$z-introd., $\exists$z-elim. disch. (b),
$\exists\alpha$-elim. disch. (a) we complete the proof.
\end{proof}

\begin{lemma}
Let $\mathrm{t(x)}$ be any term of $\,\mathrm{\bf S_2}$ with
$\alpha$ not free in it. Then
$$
\mathrm{\vdash _1 \exists ! \alpha\, \forall x\, [\,t(x) =
\alpha (x)\,]' }.
$$
\end{lemma}

\begin{proof}
By Lemma 6.1, $\;\,\mathrm{\vdash _1 \exists \alpha \, \forall x\,
[\,t(x) = \alpha (x)\,]' }$, so we assume $\mathrm{(a)\; \forall x
\,[\,t(x) = \alpha (x)\,]'}$ and, for the uniqueness,
$\mathrm{(b)\;\forall x \,[\,t(x) = \beta (x)\,]'}$, from which by
$\forall$-eliminations we get $\mathrm{(c)\; \,[\,t(x) = \alpha
(x)\,]'}$ and $\mathrm{ (d)\; \,[\,t(x) = \beta (x)\,]'}.$ By
Lemma 6.2 with (c), (d) and the \textsc{Remark}, we get
 $\mathrm{\alpha (x) = \beta (x)}$, so
$\mathrm{\forall x \,\alpha (x) = \beta (x) }$ (from (a), (b)), so
$\mathrm{(e)\;\; \alpha = \beta}.$ So from (a), (b), (e) with
$\rightarrow$-introd. disch. (b), we get
$$
\mathrm{\forall x \,[\,t(x) = \alpha (x)\,]'
~\&~\forall \beta \,[\,\forall x \,[\,t(x) = \beta (x)\,]'
\rightarrow \alpha = \beta \,]\,}.
$$
So after $\exists\alpha$-introd. and $\exists \alpha$-elim. disch.
(a) we complete the proof.
\end{proof}

For the treatment of the axioms not included in the systems of
\cite{JRMPhD}, the case of the axioms REC is trivial, since there
are no $\lambda$-occurrences in them, and the case of the axiom
schema $\mathrm{QF_t\text{-}AC_{00}}$ is obtained as follows.
\begin{lemma} Let
$$
\mathrm{E \equiv
\forall x \exists y\,
t(\langle x, y \rangle ) = 0
 \rightarrow \exists \alpha \forall x \,
t(\langle x, \alpha (x) \rangle ) = 0}
$$
where x is free for y in $\mathrm{t(\langle x, y \rangle )}$ and
$\alpha$ does not occur in $\mathrm{t(\langle x, y \rangle )}$ be
an instance of $\,\mathrm{QF _t\text{-}AC_{00}}$ in $\mathrm{\bf
S_2}$. Then $\mathrm{\vdash _1 E'}$.
\end{lemma}

\begin{proof}
We have that
$$
\mathrm{E' \equiv
\forall x \exists y\,
\,[\,t\,(\langle x, y \rangle ) = 0 \,]'\,
 \rightarrow \exists \alpha \forall x \,
\,[\,t(\langle x, \alpha (x) \rangle ) = 0\,]'}.
$$
By Lemma 6.1, $(a)\;\mathrm{\vdash _1 \exists \beta\,
 \forall w\, [\,t(w) =
\beta (w)\,]' }$, so we assume $ \mathrm{(a1)\;
 \forall w \,[\,t(w) =
\beta (w)\,]' } $ from which we get by $\forall$-elim. with the
\textsc{Remark}
$$
\mathrm{(a2)\;\;\,[\,t(\langle x, y \rangle ) =
 \beta (\langle x, y \rangle ) \,]'\,}.
$$
We assume now $\mathrm{(b)\; \forall x \exists y[\,t(\langle x, y
\rangle ) = 0 \,]',}$ so we get $\mathrm{(c)\; \exists
y\,[\,t(\langle x, y \rangle ) = 0 \,]'\,} $ and we may assume $
\mathrm{(c1)\;\,[\,t(\langle x, y \rangle ) = 0 \,]'\,}.$ By (c1),
(a2), the \textsc{Remark} and Lemma 6.2, we have $\mathrm{\beta
(\langle x, y \rangle ) = 0 }$ so we get $ \mathrm{(d)\; \exists y
\,\beta (\langle x, y \rangle ) = 0 }$ and discharge (c1) with
$\exists$y-elim. (from (c)).  With $\forall$x-introd. to (d) (x
not free in (a1), (b)),  we get $ \mathrm{(e)\;\forall x \exists y
\,\beta (\langle x, y \rangle ) = 0}$. Now by applying $\mathrm{QF
_t\text{-}AC_{00}}$ to (e) we get
$$
\mathrm{(f)\;\;\exists \alpha \forall x
 \,\beta (\langle x, \alpha (x) \rangle ) = 0 }.
$$
We assume now $ \mathrm{(g)\;\forall x
 \,\beta (\langle x, \alpha (x) \rangle ) = 0 }$ from which we get
$ \mathrm{(g1)\;\beta (\langle x, \alpha (x) \rangle ) = 0 }$.
With $\forall$w-elim. from (a1) and the \textsc{Remark} we get
$\mathrm{\;[\,t (\langle x, \alpha (x) \rangle ) = \beta (\langle
x, \alpha (x) \rangle )\,]' } $, so by (g1) and the
\textsc{Remark} after $\forall$-introd. (x not free in (a1), (b),
(g)), we have $ \mathrm{(h)\;\forall x
 \,[\,t (\langle x, \alpha (x) \rangle ) = 0\,]' },$
and finally
$$
\mathrm{(i)\;\;\exists \alpha \forall x
\,[ \,t (\langle x, \alpha (x) \rangle ) = 0 \,]'\,},
$$
with $\exists\alpha$-introd. and $\exists \alpha$-elim. disch. (g)
(from (f)). After completing the $\exists \beta$-elim. disch. (a1)
(from (a)), with $\rightarrow$-introd. disch. (b)  we get
$\mathrm{\vdash _1 E'}$.
\end{proof}

\begin{theorem}
$\mathrm{(a)}$ $\mathrm{{\bf EL}}$ is a definitional extension of
$\mathrm{{\bf EL}-\lambda}$.

$\mathrm{(b)}$ $\mathrm{{\bf IA_1} + QF\text{-}AC_{00}}$ is a
definitional extension of $\mathrm{{\bf IA_1} +
QF\text{-}AC_{00}-\lambda}$.
\end{theorem}

\section{Comparison of {\bf M} with {\bf BIM}, {\bf H} and {\bf WKV}}

\subsection{The formal system {\bf BIM}}

\subsubsection{}

The formal system {\bf BIM} of \emph{Basic Intuitionistic
Mathematics} (described variously in \cite{Veldman1},
\cite{Veldman2}, \cite{Veldman3}) has been introduced by W.
Veldman to serve as a formal basis for intuitionistic mathematics.
This system has a very small collection of primitives: the most
recent version in \cite{Veldman3} includes constants for 0, for
the unary function with constant value 0, for the identity
function, for the successor, and constants J, K, L, for a binary
pairing function (which is taken to be onto the natural numbers)
with the corresponding projection functions. Besides axioms
concerning equality, the non-logical axioms comprise the axiom
schema of mathematical induction, finitely many axioms on the
function constants, presented as a single axiom,  and axioms
guaranteeing the closure of the set of (one-place
number-theoretic) functions under the operations of composition,
pairing and primitive recursion, again presented as conjuncts of a
single axiom. Constants for primitive recursive functions and
relations can be added conservatively. There is no
$\lambda$-abstraction. The following axiom of \emph{unbounded
search} is also included, as the last conjunct of the axiom on the
closure of the set of functions:
$$
\mathrm{\forall \alpha \,[\,\forall m \exists n\, \alpha (m, n)=0
\rightarrow \exists\gamma \forall m\,[\, \alpha (m, \gamma (m)) = 0
~\&~ \forall n < \gamma (m) \,\alpha (m,n) \neq 0 \,]\,],}
$$
where $\mathrm{\alpha (m,n)}$ abbreviates $\mathrm{\alpha
(J(m,n))}$,\footnote{Note that in the absence of more function
constants $\mathrm{x<y}$ can be introduced as abbreviating
$\mathrm{\exists z ( S(z) + x = y)}$ just by adding + (addition)
to the formalism.} which guarantees closure under the unbounded
least number operator and through its consequence
$$
 \mathrm{\forall \alpha \,[\,\forall m \exists n\, \alpha (m,
n)=0 \rightarrow \exists\gamma \forall m\, \alpha (m, \gamma
(m)) = 0\,]}
$$
(called \emph{the minimal axiom of countable choice}) expresses,
again in the form of a single axiom, countable numerical choice
for a special case of quantifier-free formulas.

\begin{proposition}
Over {\bf BIM}, the schema $\mathrm{CF\!_d}$ entails
$\mathrm{AC_{00}!}$.
\end{proposition}

\begin {proof}
The proof is similar to that of Theorem 3.2 since {\bf BIM} proves
the decidability of number equality (by the axiom schema of
induction). We only have to add (conservatively) addition + to
{\bf BIM} so that $<$ can be expressed efficiently, and use the
constants J, K, L instead of the pairing and projection functions
of $\mathrm{\bf IA_1}$ and the axiom of unbounded search instead
of $\mathrm{QF\text{-}AC_{00}}$.
\end{proof}

In fact, over {\bf BIM} without the axiom of unbounded search, the
conjunction of the axiom of unbounded search with the schema
$\mathrm{CF\!_d}$ is equivalent to $\mathrm{AC_{00}!}$. The one
direction is given by the preceding proposition, and for the
converse, $\mathrm{AC_{00}!}$ entails $\mathrm{CF\!_d}$ like in
Proposition 3.1, and it is easy to see that $\mathrm{AC_{00}!}$
entails the axiom of unbounded search adapting Proposition 2.7,
(again by adding addition + to {\bf BIM} so that $<$ can be
expressed efficiently).

\subsubsection{}

The result of \cite{JRMPhD} on  $\lambda$-eliminability applies to
any formal system of the sort that we are studying which has at
least function constants  $0,\, ', +, \cdot , \mathrm{exp}$, and
assumes $\mathrm{AC_{00}!}$. The system {\bf BIM} can be extended
definitionally to a system {\bf BIM$'$}, so that it includes all
these constants. Consequently, the $\lambda$ symbol with the
related formation rule and axiom schema can be added
definitionally to the system {\bf BIM$'$} + $\mathrm{CF\!_d}$.
Consider now the system obtained by the addition of $\lambda$,
{\bf BIM$'\,$}+$\,\mathrm{CF\!_d}\,$+$\,\lambda$. In this system,
thanks to the presence of $\lambda$, we can easily obtain  Lemma
5.3(b) of [FIM] from the following \emph{Axiom of Primitive
Recursion} of {\bf BIM}:
$$
\mathrm{\forall \alpha \forall \beta\exists \gamma \forall m
\forall n \,[\,\gamma (m, 0)=\alpha (m) ~\&~\gamma (m, S(n))=
\beta (m, n,\gamma (m, n))\,]\,.}
$$
By the same method that the recursor constant rec is added to {\bf
M}, all the additional function and functional constants of {\bf
M} can be added definitionally, successively according to their
primitive recursive descriptions, to this extended system. Let
$\mathrm{{\bf BIM^+}}$ be the result of adding to $\mathrm{{\bf
BIM}}$ $\lambda$ and all the additional constants of {\bf M}, with
their formation rules and axioms. Let {\bf M$^j$} be the trivial
extension of {\bf M} obtained by adding the constants J, K, L
(note that there exists a primitive recursive pairing function
which is onto the natural numbers and has primitive recursive
inverses, and of course all three can be added definitionally to
{\bf M}, to serve as interpretations for J, K, L). Then
$\mathrm{{\bf BIM^+}}$ + $\mathrm{CF\!_d}$ and {\bf M$^j$}
coincide up to trivial differences, and we have the following:

\begin{theorem}
$\mathrm{{\bf M^j}}$ is a definitional extension of
$\;\mathrm{{\bf BIM}} + \,\mathrm{CF\!_d}$.
\end{theorem}

We conclude that {\bf BIM} + $\mathrm{CF_d}$ and {\bf M} are
essentially equivalent, and also essentially equivalent with {\bf
EL} + $\mathrm{CF_d}$. It is remarkable that, in developing the
theory within {\bf BIM}, W. Veldman defines a set of natural
numbers to be decidable if it has a characteristic function, as
this means that $\mathrm{CF\!_d}$ is implicitly assumed.

\begin{theorem}
$\;\mathrm{{\bf BIM}}$ does not prove $\,\mathrm{CF\!_d}$.
\end{theorem}

\begin{proof}
This is obtained as in the case of {\bf EL}, Thm. 3.4,
interpreting the function variables as varying over the general
recursive functions of one number variable.
\end{proof}

\subsection{The formal system {\bf H}}

In \cite{Howard-Kreisel}, the formal system of \emph{elementary
intuitionistic analysis} {\bf H} is used. {\bf H} differs from
{\bf BIM} essentially only in that it does not assume the axiom of
unbounded search (some primitive recursive functions and relations
are introduced in the system from the beginning, but this is an
inessential difference).

Like {\bf IA$_1$} and  {\bf HA$_1$}, {\bf H} has a classical model
consisting of the primitive recursive functions. Essentially {\bf
H} is a proper subtheory of {\bf BIM}.

\subsection{The formal system {\bf WKV}}

In \cite{Loeb} Iris Loeb presents and uses the formal system {\bf
WKV} (for ``Weak Kleene-Vesley''). The constants include 0, S, +,
$\cdot$, =, j (for a pairing function which is onto) and j$_1$,
j$_2$ (projections). This system has $\lambda$-abstraction. Among
the axioms assumed, there is an axiom schema of primitive
recursion in the version
$$
\mathrm{\exists \beta \,[\, \beta (0) = t ~\&~ \forall y\,
\beta (S(y)) = r(j(y, \beta (y)))\,]\,},
$$
where t is a term and r a functor, and $\mathrm{AC_{00}!}$. We
observe that the assumed version of the axiom schema of primitive
recursion is very similar to  Lemma 5.3(b) of [FIM]; in fact, as
it is easily seen, these schemas are equivalent (modulo the
pairing) over both {\bf WKV} and {\bf M}. Using the method of the
addition of the recursor constant rec to {\bf M}, all constants of
{\bf M} can, successively according to their primitive recursive
descriptions, be added definitionally to {\bf WKV}. So, if we
consider the trivial extension {\bf M$^j$} of {\bf M} by the
pairing and projections of {\bf WKV} with their axioms, we
conclude that the two systems are essentially equivalent.

\begin{theorem}
$\mathrm{{\bf M^j}}$ is a definitional extension of
$\,\mathrm{{\bf WKV}}$.
\end{theorem}

\section{Concluding observations}

\subsection{More on {\bf BIM} and {\bf H}}

The following clarify, among other things, the relationship
between {\bf EL} and {\bf BIM}.

\textsc{Observation 1.} In the presence of enough function
constants in {\bf BIM} without the axiom of unbounded search and
in {\bf H}, the single minimal axiom of countable choice is
equivalent to QF-AC$_{00}$. To see this, consider the description
of {\bf BIM} in \cite{Veldman3}. What follows holds also for {\bf
H}. As it is noted, thanks to the presence of the pairing
function, the binary, ternary etc. functions can be treated as
unary. So we can consider (definitional) extensions of {\bf BIM}
by only unary function constants, expressing primitive recursive
functions, introduced as functors in the formalism (see below the
treatment of the recursor R); then, since the only non-unary
function constant is the binary constant J for the pairing
function, for which we get from the axioms $(*)\;\mathrm{\forall
\alpha \forall \beta \exists \gamma \forall n \gamma (n) =
J(\alpha (n) , \beta (n))}$, we have:

(a) For every term t with free number variables say for simplicity
the x, y, we can prove $\mathrm{ \;  \exists \alpha (\alpha(x,y) =
t (x,y))}$ (for t(x)$\equiv$x consider $\alpha$ with
$\mathrm{\alpha (x)=J(K(x), L(x))}$, justified by $(*)$ and the
axioms for J, K, L, or the identity function, absent from the
earlier versions of the system).

If an extension as above contains a sufficient supply of primitive
recursive functions, essentially like in \cite{FIM}, pp. 27, 30,
we have:

(b) For any quantifier-free formula A(x, y, z) with free number
variables the indicated three (for example), we can prove
$\mathrm{\;\exists \gamma \forall x \forall y \forall z
\,(A(x,y,z) \leftrightarrow \gamma (x,y,z) = 0)}$.

Moreover:

(c) By formal induction on z, we can prove $\mathrm{\forall z
\exists \beta \forall n \beta (n) =z}$.

(d) Given a (number) z and a (function) $\gamma$, we may assume,
by (c), $\mathrm{\forall n \beta (n) =z}$ and define
$\mathrm{\delta (w) = \gamma (K(w), L(w), z)=\gamma (K(w), L(w),
\beta (w))}$, using $(*)$.

Now let $\mathrm{\forall x \exists y A(x,y,z)}$ be the hypothesis
of an instance of QF-AC$_{00}$. By the above facts we get a
$\delta$ such that $\mathrm{A(x,y,z) \leftrightarrow \delta (x,y)
= 0}$, then get $\mathrm{\forall x \exists y \delta (x,y)=0}$, and
applying the minimal axiom of countable choice to this get
$\mathrm{\exists \alpha \forall x \delta (x, \alpha (x)) =0}$ and
finally $\mathrm{\exists \alpha \forall x A(x, \alpha (x), z)}$.

\textsc{Observation 2.} We can add definitionally to {\bf H} and
{\bf BIM} (and in their extensions as above, but in this case the
previous additions become redundant) a recursor constant: we add a
constant R together with the formation rule ``if t is a term and u
a functor, then R(t,u) is a functor'' and defining axiom $\mathrm{
A( x, \alpha , R(x, \alpha ))}$, where $\mathrm{A(x, \alpha ,
\beta )}$ is the formula
$$
\mathrm{ \forall y \forall z [\,\beta (y,0) = x ~\&~
\forall z \beta (y, S(z)) = \alpha (y, z, \beta (y, z))\,]\,}.
$$
We note that facts (a) - (d) hold in the present extension too.
Since we have $\mathrm{\vdash \exists ! \beta A(x, \alpha , \beta
)}$, we can show that this extension is definitional, using the
translation $'$ defined as follows, with terminology, notation and
conditions on the variables as in the previous cases, and with use
of functional versions of results from \cite{IM}, see discussion
before Lemma 4.7.\footnote{For the treatment of definitional
extensions by the introduction of functors in the formalism, see
\cite{Kreisel-Troelstra}. Note that the system called {\bf EL}
(and its subsystem {\bf EL}$_0$) in that work, assumes
AC$_{00}!$.} For each prime formula E without R, E$'$ is E.
Otherwise, if R(t,u) is the first occurrence of an R-plain functor
in E, then
$$
\mathrm{ E' \equiv \exists \beta \,[\, A(t, u, \beta ) ~\&~
[\,E(\beta )\,]'\,]\,}.
$$

\textsc{Observation 3.} In Lemmas 6.1 - 6.3 (on
$\lambda$-eliminability) there is just one use of QF-AC$_{00}$,
for the basis of the inductive argument of Lemma 6.1. By the facts
(a), (d) above, this use can be avoided. So Lemmas 6.1 - 6.3 are
valid for suitable extensions (which may include the recursor R)
of {\bf BIM} and {\bf H}, allowing to add $\lambda$
definitionally.

We have now the following conclusions (note that R can be
introduced in {\bf HA$_1$} and {\bf EL} by explicit definition).

\begin{theorem}
The systems $\mathrm {{\bf H}}$ and $\mathrm{{\bf HA}_1}$ have a
common conservative extension obtained by definitional extensions,
so they are essentially equivalent.
\end {theorem}

\begin{theorem}
The systems $\mathrm {{\bf BIM}}$ and $\mathrm{{\bf EL}}$ have a
common conservative extension obtained by definitional extensions,
so they are essentially equivalent.
\end{theorem}

\subsection{Additional results}

In proving that the recursor constant rec can be added
definitionally to {\bf M}, AC$_{00}$! has been used in two cases
(except from the point where it is shown that the translation of
AC$_{00}$! itself is a theorem of {\bf M}): in the proof of Lemma
5.3(b) of [FIM], and in the proof of Lemma 4.4. We will prove
these two lemmas by using only QF-AC$_{00}$. In this way, together
with the observation that  QF-AC$_{00}$ is equivalent over {\bf
IA$_1$} with the single axiom
$$
\mathrm{\forall \alpha
\,[\,\forall x \exists y\, \alpha (\langle x, y\rangle)=0
\rightarrow \exists\gamma \forall x\, \alpha (\langle x, \gamma
(x)\rangle) = 0 \,]\,},
$$
we obtain the fact that rec can be added definitionally  to {\bf
IA$_1$} + QF-AC$_{00}$.

\begin{lemma}
In $\,\mathrm{{\bf IA_1} + QF\text{-}AC_{00}}$,
$$
\mathrm{\vdash  \exists  \beta  \,[\, \beta (0) = x ~\&~
\forall z \,\beta (z')=\alpha (\langle \beta (z),z \rangle
)\,]\,}.
$$
\end{lemma}

\begin{proof}
We slightly modify the proof of Lemma 5.3(b) of [FIM] as follows.
Let
$$
\mathrm{P(x, \alpha , y, v)\equiv (v)_0 = x  ~\&~ \forall i < y
\,(v)_{i'} = \alpha (\langle (v)_i, i \rangle)}.
$$
By formal induction (IND) on y we show first $\mathrm{(a)\;\;
\vdash \forall y \exists v\, P(x, \alpha , y, v)}$ as follows:

\textsc{Basis.} We get $\mathrm{\exists v\, P(x, \alpha , 0,v)}$
by ``setting'' $\mathrm{v = p_0^x}$.

\textsc{Ind. step.} Assuming $\mathrm{ P(x, \alpha , y, w)}$ and
``setting'' $\mathrm{ v = \Pi_{i\leq y} p_i^{(w)_i}*p_{y'}^
{\alpha (\langle (w)_y, y \rangle )}},$ since then
$$
\mathrm{\forall i\leq y \;(v)_i = (w)_i ~\&~
(v)_{y'} = \alpha (\langle (w)_y, y \rangle )},
$$
we get $\mathrm{\exists v\, P(x, \alpha , y', v)}$ from the
inductive hypothesis.

Using f$_{15}$, $\mathrm{P(x, \alpha , y, v)}$ is equivalent over
{\bf IA$_1$} to a quantifier-free formula (with the same free
variables). So we can apply QF-AC$_{00}$ to (a) and get
$$
\mathrm{(b)\;\; \exists \gamma \forall y \, P(x, \alpha , y,
\gamma (y))}.
$$
Assume now $\mathrm{(c)\;\;  \forall y \, P(x, \alpha , y, \gamma
(y))}.$ We ``define'' now
$$
\mathrm{\beta = \lambda y. \gamma (y)_y}
$$
(justified by Lemma 5.3(a) of [FIM]). Then we can show
$$
\mathrm{(d)\;\;\exists  \beta  \,[\, \beta (0) = x ~\&~
\forall z \,\beta (z')=\alpha (\langle \beta (z),z \rangle
)\,]\,}
$$
(for the second conjunct in (d), specialize from (c) for z$'$ and
$\mathrm{i = z \;(z< z')}$, and get $\mathrm{ \gamma (z')_{z'} =
\alpha (\langle \gamma (z')_z,z \rangle )}$. But $\mathrm{ \gamma
(z')_z = \gamma (z)_z}$ (by formal induction (IND) on z, using
(c), we can show $\mathrm{\forall z \forall i \leq z \, \gamma
(z')_i = \gamma (z)_i}$), so $\mathrm{ \beta (z') = \alpha
(\langle \beta (z),z \rangle )}$).
\end{proof}

\begin{lemma}
In $\,\mathrm{{\bf IA_1} + QF\text{-}AC_{00}}$,
$$
\mathrm{\vdash  \exists ! \beta  \,[\, \beta (0) = x ~\&~
\forall z \,\beta (z')=\alpha (\langle \beta (z),z \rangle
)\,]\,}.
$$
\end{lemma}

In the next lemma  the notation and abbreviations are as in Lemma
4.4.

\begin{lemma}
Let $\mathrm{t, s}$ be terms and $\mathrm{u}$ a functor of
$\,\mathrm{{\bf IA_1}}$. Let $\mathrm{x_0, \ldots , x}_k$ include
all the number variables occurring free in $\mathrm{t, u
\;\text{or} \;s}$, let $\mathrm{w}$ and $\mathrm{v}$ be distinct
number variables not occurring in $\mathrm{t, u, s}$,  and
$\gamma$ a function variable free for $\mathrm{v}$ in
$\mathrm{A(t^w, u^w, s^w, v)}$, not occurring free in
$\mathrm{A(t^w, u^w, s^w, v)}$. Then in $\,\mathrm{{\bf IA_1} +
QF\text{-}AC_{00}}$
$$
\mathrm{\vdash  \exists ! \gamma \forall w A(t^w,
u^w, s^w, \gamma (w))}.
$$
\end{lemma}

\begin{proof}
We want to show
$$
\mathrm{\vdash  \exists ! \gamma \forall w \exists \beta
\,[\,\beta (0) = t^w ~\&~ \forall z \,\beta (z')=(u^w) (\langle
\beta (z),z \rangle ) ~\&~ \beta (s^w) = \gamma (w)) \,]\,}.
$$
From (a) in the proof of Lemma 8.3 we have $\mathrm{\vdash \forall
w \exists v\, P(t^w, u^w , s^w, v)}$. Applying QF-AC$_{00}$ we get
$\mathrm{(a)\;\; \vdash \exists \delta \forall w \,P(t^w, u^w ,
s^w, \delta (w))}$. Assume
$$
\mathrm{(a1)\;\; \forall w \,[\,(\delta (w))_0 = t^w
~\&~ \forall i < s^w (\delta (w))_{i'} = (u^w) (\langle (\delta
(w))_i, i \rangle)\,]\,},
$$
and let $\mathrm{(*)\;\;\gamma = \lambda w. (\delta (w))_{s^w}}.$
By (a1), we get
$$
\mathrm{(a2)\;\; \,(\delta (w))_0 = t^w
~\&~ \forall i < s^w (\delta (w))_{i'} = (u^w) (\langle (\delta
(w))_i, i \rangle)}.
$$
From Lemma 8.4,
$$
\mathrm{(b)\;\;\vdash  \forall w \exists ! \beta  \,[\, \beta (0) = t^w ~\&~
\forall z \,\beta (z')=(u^w) (\langle \beta (z),z \rangle
)\,]\,},
$$
so we may assume
$$
\mathrm{(b1)\;\; \beta (0) = t^w ~\&~
\forall z \,\beta (z')=(u^w) (\langle \beta (z),z \rangle )}.
$$
Then we get $\mathrm{(c)\;\;\beta (s^w) = (\delta (w))_{s^w}}$ (to
get (c) we prove  $\mathrm{ \forall i \leq s^w \;\beta (i) =
(\delta (w))_i}$  by formal induction (IND) on i), so
$\mathrm{(d)\;\gamma (w) = \beta (s^w)}$. By $\rightarrow$-elim.
disch. (b1) with $\forall\beta$- and $\forall$w-introds. we get
$\mathrm{(e)\;\forall w \forall \beta \,[\,(b1) \rightarrow \gamma
(w) = \beta (s^w)\,].}$ From (b) and (e) we get
$$
\mathrm{(f)\;\;\forall w   \exists  \beta \,[\,\beta (0) = t^w
~\&~ \forall z \,\beta (z')=(u^w) (\langle \beta (z),z \rangle )
~\&~ \beta (s^w) = \gamma (w) \,]\,}.
$$
For the uniqueness, assume now
$$
\mathrm{(g)\;\;\forall w   \exists  \beta \,[\,\beta (0) = t^w
~\&~
\forall z \,\beta (z')=(u^w) (\langle \beta (z),z \rangle
) ~\&~ \beta (s^w) = \varepsilon (w) \,]\,}.
$$
Then from (e) and (f) we get easily $\mathrm{(h)\;\forall
w\,\varepsilon (w) = \gamma (w)}.$ So from (g) we get
$$
\mathrm{(i)\; \;\varepsilon = \gamma },
$$
and then $\forall \varepsilon$ ((g) $\rightarrow$ (i)) (disch. (g)
before $\forall$-introd.). So with (f) and $\exists\gamma$-introd.
we get the lemma, after completing the
 $\exists \gamma$,  $\exists \delta$-elims. disch.
 $(*)$ and (a1), respectively.
\end{proof}

\begin{corollary}
The system $\,\mathrm{{\bf IA_1} + QF\text{-}AC_{00}\;+ \;Rec}$ is
a definitional extension of $\,\mathrm{{\bf IA_1} +
QF\text{-}AC_{00}}$.
\end{corollary}

\subsection{Obtained relations}

Below we  collect the main results following from the arguments we
presented. We note that II(b) justifies completely the observation
of J. Rand Moschovakis (\cite{JRM2014}, Thm. 1) that {\bf BIM},
{\bf EL} and {\bf IA$_1$} + QF-AC$_{00}$ can be used
interchangeably as a basis for intuitionistic reverse analysis.

\vskip 0.2cm

I. In relation to the small classical models that the weak
constructive systems we studied admit:

(a) The systems {\bf H}, {\bf IA$_1$}, {\bf HA$_1$}, {\bf IA$_1$}
+ Rec have a classical model consisting of primitive recursive
functions.

(b) {\bf BIM} and all the systems resulting from adding
QF-AC$_{00}$  to {\bf IA$_1$}, {\bf HA$_1$}, {\bf IA$_1$} + Rec
have a classical model consisting of general recursive functions.

(c) Adding $\mathrm{CF\!_d}$ to the systems of (b) gives stronger
systems, which do not admit small classical models consisting of
recursive functions.

\vskip 0.2cm

II. The systems of each group are proof-theoretically essentially
equivalent:

(a) {\bf IA$_1$} + Rec, {\bf HA$_1$}, {\bf H}.

(b) {\bf IA$_1$} + QF-AC$_{00}$, {\bf EL}, {\bf BIM}, {\bf H} +
\textit{Unbounded Search}.

(c) {\bf M}, {\bf WKV}, {\bf EL} + $\mathrm{CF\!_d}$, {\bf BIM} +
$\mathrm{CF\!_d}$, {\bf H} + AC$_{00}!$.

\vskip 0.2cm

III. We observe also the following.

(a) Adding QF-AC$_{00}$ to any of $\,${\bf H}, {\bf IA$_1$}, {\bf
HA$_1$} gives stronger systems.

\vskip 0.1cm

(b) Adding QF-AC$_{00}$ to any of $\,${\bf H} + $\mathrm{CF\!_d}$,
{\bf IA$_1$} + $\mathrm{CF\!_d}$, {\bf HA$_1$} + $\mathrm{CF\!_d}$
gives stronger systems (J. Rand Moschovakis, \cite{JRM-GV2012}).

\vskip 0.1cm

(c) Adding $\mathrm{CF\!_d}$ to any of $\,${\bf H}, {\bf IA$_1$},
{\bf HA$_1$}, {\bf EL}, {\bf BIM} gives stronger systems.

\end{document}